\documentclass[a4paper, reqno]{amsart}
\usepackage{amssymb, amsmath, amsthm, chngcntr, enumerate, mathabx, mathrsfs, mathtools, pgf, tikz,esint}
\usepackage{graphicx}
\usepackage{dsfont}
\usetikzlibrary{arrows,calc}
\usepackage[english]{babel}
\usepackage[T1]{fontenc}

\numberwithin{equation}{section}
\newtheorem{theorem}{Theorem}[section]

\newtheorem{lemma}[theorem]{Lemma}
\newtheorem{proposition}[theorem]{Proposition}

\theoremstyle{definition}

\theoremstyle{definition}
\newtheorem{remark}[theorem]{Remark}
\newtheorem{example}{Example}

\newtheorem{definition}[theorem]{Definition}
\newtheorem{definition lemma}[theorem]{Definition/Lemma}

\newcommand{\R}{\mathbb{R}}
\newcommand{\Z}{\mathbb{Z}}

\newcommand{\N}{\mathbb{N}}

\renewcommand{\P}{\mathbb{P}}

\newcommand{\ud}{\mathrm{d}}

\newcommand{\size}{\mathrm{Size}}
\newcommand{\energy}{\mathrm{Energy}}
\newcommand{\eps}{\varepsilon}
\newcommand{\VarC}[1]{\mathrm{Var}^{#1}\mathscr{C}}

\counterwithout{equation}{section}
\counterwithout{section}{section}

\author{Fr\'{e}d\'{e}ric Bernicot \and Marco Vitturi}

\address{CNRS - Universit\'e de Nantes \\ Laboratoire Jean
Leray \\ 2, rue de la Houssini\`ere
44322 Nantes cedex 3, France}
\email{frederic.bernicot@univ-nantes.fr \and marco.vitturi@univ-nantes.fr }

\title[Bilinear RdF with non-smooth squares]{Bilinear Rubio de Francia inequalities for collections of non-smooth squares}
\keywords{Bilinear Fourier multipliers, orthogonality}
\subjclass{42A45}

\begin{document}

%
% abstract
%
\begin{abstract}
Let $\Omega$ be a collection of disjoint dyadic squares $\omega$, let $\pi_\omega$ denote the non-smooth bilinear projection onto $\omega$
\[ \pi_\omega (f,g)(x):=\iint \mathds{1}_{\omega}(\xi,\eta) \widehat{f}(\xi) \widehat{g}(\eta) e^{2\pi i (\xi + \eta) x} \ud \xi \ud\eta \]
and let $r>2$. We show that the bilinear Rubio de Francia operator
\[ \Big(\sum_{\omega\in\Omega} |\pi_{\omega} (f,g)|^r \Big)^{1/r} \] 
is $L^p \times L^q \rightarrow L^s$ bounded with constant independent of $\Omega$ whenever $1/p + 1/q = 1/s$, $r'<p,q<r$, and $r'/2 < s < r/2$.
\end{abstract}

\maketitle
%
  
%
%
%
% INTRODUCTION 
\section{Introduction}

Classical Littlewood-Paley theory on the real line is a staple of linear harmonic analysis and has proven vastly important in its development. It encodes a principle of orthogonality in $L^p$ spaces even when $p\neq 2$ for dyadically separated frequencies, and can thus be seen as a substitute for Plancherel's identity; this usually allows one to decouple the action of a multiplier on each dyadic frequency and deal with them separately. Generalizations of the linear Littlewood-Paley inequalities were first considered by Carleson in \cite{Carleson_LP} (later reproved in a different way by Cordoba in \cite{Cordoba}) for the special case where one replaces the Littlewood-Paley dyadic intervals $[2^k,2^{k+1}], k\in\Z$ by the intervals $[n,n+1], n\in\Z$. Later, Rubio de Francia in \cite{RubioDeFrancia} extended Carleson's result to arbitrary collections of disjoint intervals. In particular, he proved the following: let $\mathcal{I} = \{I_j\}_j$ be a collection of disjoint intervals and define the Rubio de Francia square function
\[ \mathrm{RdF}^2_{\mathcal{I}}f(x) := \Big(\sum_{j} |\pi_{I_j} f(x)|^2 \Big)^{1/2},  \]
where $\pi_{I}$ is the frequency projection operator defined by
\[ \widehat{\pi_{I}f}(\xi) := \mathds{1}_{I}(\xi) \widehat{f}(\xi); \]
then for all $2\leq p < \infty$ it holds that for all $f\in L^p(\R)$ 
\begin{equation}\label{eqn:RdF_inequality} 
\|\mathrm{RdF}^2_{\mathcal{I}}f\|_{L^p(\R)} \lesssim_p \|f\|_{L^p(\R)} 
\end{equation}
(with constant independent of $\mathcal{I}$). The inequality is false in general for $p<2$, as was known since \cite{Carleson_LP} - this corresponds to a failure of orthogonality in $L^p$ spaces for small $p$'s. More in general, by the same methods one can prove for a generic $r>2$ that the Rubio de Francia $r$-function
 \[ \mathrm{RdF}^r_{\mathcal{I}}f(x) := \Big(\sum_{j} |\pi_{I_j} f(x)|^r \Big)^{1/r} \]
 is bounded on $L^p$ for all $r'< p < \infty$ (the lowerbound being sharp; see \cite{TaoCowling} for a proof). The condition $r\geq 2$ is necessary, as can be seen for example by considering the collection of Littlewood-Paley intervals. Known proofs of \eqref{eqn:RdF_inequality} (see \cite{Journe}, \cite{Sjoelin},  \cite{Soria}, \cite{Lacey_RdF}) rely on an interpolation between the trivial $L^2$ case and (a substitute for) the $L^\infty$ endpoint (or dually between $L^2$ and $H^1$, as in \cite{Bourgain}). See also \cite{BeneaMuscalu_preprint} about an alternative proof for such inequalities as well as for a bilinear generalization, involving a collection of paraproducts-type operators. Higher dimensional versions of the inequalities have been first shown in \cite{Journe}.\\
 A natural question is whether similar orthogonality principles exist in the bilinear setting and to what extent. That is, given bilinear multiplier operators $T_j$ with disjoint frequency supports in the frequency plane $\widehat{\R^2}$, under what conditions does it hold that, say, the square function
 \[ \Big(\sum_{j} |T_j (f,g)|^2\Big)^{1/2} \]
 is bounded from $L^p \times L^q$ to $L^s$? Some results are known for special collections of supports. Perhaps the first one is to be found in Lacey's \cite{Lacey_bilinearSqF}, where he proves the $L^p \times L^q \rightarrow L^2$ boundedness of the bilinear square function
 \[ f,g \mapsto \Big(\sum_{n \in \Z} \Big|\iint \chi(\xi - \eta - 2n) \widehat{f}(\xi) \widehat{g}(\eta) e^{2\pi i (\xi + \eta) x} \ud \xi \ud \eta\Big|^2 \Big)^{1/2} \]
 for $p,q\geq 2$ such that $1/p + 1/q = 1/2$ (later extended to any $1/p + 1/q = 1/s$ in \cite{MohantyShrivastava},\cite{BernicotShrivastava}), where $\chi$ is a $C^\infty$ function that is identically $1$ in $[-1/2,1/2]$ and vanishes outside $[-1,1]$. Thus here the frequency supports consist of (smoothened) diagonal strips of roughly unit width and unit separation. This was later extended by the first author in \cite{Bernicot} to the case of non-smooth diagonal strips, that is where one replaces the smooth function $\chi$ above by the non-smooth $\mathds{1}_{[-1/2,1/2]}$. The discontinuity at the boundary of the strip makes the analysis inherently more complicated (the same phenomenon that arises in the study of, for example, the Bilinear Hilbert transform).\\
A further and more recent example of a bilinear $r$-function of the form above is given by 
\[ f,g \mapsto \Big(\sum_{j \in \Z} \Big|\iint\limits_{a_j < \xi < \eta < a_{j+1} } \widehat{f}(\xi) \widehat{g}(\eta) e^{2\pi i (\xi + \eta)x }\ud \xi \ud \eta \Big|^r\Big)^{1/r} \]
for a sequence of strictly increasing real numbers $a_j < a_{j+1}$. It can be thought of as a bilinear Rubio de Francia operator for iterated Fourier integrals, which finds its motivation naturally in the analysis of the stability of solutions to AKNS systems. It was proven in \cite{BeneaMuscalu} that for $r\geq 2$ this operator is $L^p \times L^q \to L^s$ bounded for the same exponents for which the Bilinear Hilbert transform is bounded (that is, under the necessary condition $1/p + 1/q = 1/s$, at least when $p,q > 1$ and $s > 2/3$); moreover, it is also bounded when $1 \leq r < 2$, although the range depends on $r$ in this case. We refer the reader to \cite{BeneaMuscalu} for details and for the aforementioned physical motivation.\\
 \par
 In this paper we are interested in bilinear operators built out of bilinear projections whose frequency supports consist of squares in the frequency plane $\widehat{\R^2}$. Here the reference we have in mind is \cite{BeneaBernicot} by Benea and the first author, in which the following bilinear versions of Rubio de Francia $r$-functions are considered: let $\Omega$ be a collection of disjoint squares in $\widehat{\R^2}$ and let $r$ be fixed, then define the operator 
 \[ S_{\Omega}^r(f,g)(x) := \Big(\sum_{\omega \in \Omega} \Big|\iint \chi_{\omega}(\xi, \eta) \widehat{f}(\xi) \widehat{g}(\eta) e^{2\pi i (\xi + \eta) x} \ud \xi \ud \eta \Big|^r \Big)^{1/r}, \]
 where $\chi_{\omega}$ is a $C^{\infty}$ function that is identically $1$ on $\frac{1}{2}\omega$ and vanishes outside $\omega$. In \cite{BeneaBernicot} the authors prove the following theorem:
 %
 % Theorem in Benea, Bernicot
 %
 \begin{theorem}[\cite{BeneaBernicot}]\label{thm:smooth_squares_theorem}
Let $\Omega$ be a collection of disjoint squares in $\widehat{\R^2}$ and let $r>2$. Then 
\begin{equation}
\| S_{\Omega}^r(f,g)\|_{L^s(\R)} \lesssim_{p,q} \|f\|_{L^p} \|g\|_{L^q}
\end{equation}
for all $p,q,s$ such that $1/p + 1/q = 1/s$, $r' < p,q < \infty$, $r'/2 < s < r$. In particular, the constant is independent of $\Omega$.
 \end{theorem}
 This result is to be thought of as a bilinear orthogonality principle for collections of (smoothened) frequency squares in the same way as the Rubio de Francia theorem is for the linear case. Observe however that the square function case $r=2$ is not covered by the theorem - its boundedness is currently an open problem. We remark that the condition $r' < p,q$ is necessary (to see why it suffices to consider a collection of squares like the one given in Example \ref{ex:line} below).
 \par
 Our interest here is to extend the results of \cite{BeneaBernicot} to the case where the smooth characteristic function $\chi_\omega$ above is replaced by the non-smooth characteristic function $\mathds{1}_\omega$. In particular, let $\Omega$ be a collection of disjoint squares $\omega = \omega_1 \times \omega_2$ in $\widehat{\R^2}$ and denote by $\pi_{\omega}$ the non-smooth bilinear frequency projection onto the square $\omega$, that is 
\[ \pi_{\omega} (f,g)(x) := \iint \mathds{1}_{\omega}(\xi,\eta) \widehat{f}(\xi)\widehat{g}(\eta) e^{2\pi i (\xi + \eta)x} \ud \xi \ud \eta, \]
which in particular factorizes as $\pi_\omega = \pi_{\omega_1} \otimes \pi_{\omega_2}$. We are interested in the bilinear operator 
\[ f,g \mapsto T_{\Omega}^r(f,g)(x):= \Big(\sum_{\omega\in\Omega} |\pi_{\omega}(f,g)(x)|^r \Big)^{1/r}, \qquad r\geq 2, \]
and specifically in proving bounds of the form
\begin{equation}\label{eqn:target_estimate}
\|T_{\Omega}^r(f,g)\|_{L^s} \leq C \|f\|_{L^p} \|g\|_{L^q};
\end{equation}
we denote by $C_{p,q,s,\Omega}$ the best constant $C$ such that the above inequality holds for all $f \in L^p, g \in L^q$ (we consider $r$ fixed). The usual scaling argument shows that a necessary condition is that the exponents $p,q,s$ satisfy H\"{o}lder's relationship, that is it must be 
\[ \frac{1}{p} + \frac{1}{q} = \frac{1}{s}\]
(and therefore $C_{p,q,s,\Omega} = C_{p,q,\Omega}$).\\
We consider some examples in order to get acquainted with the problem at hand.
%
% Example of collection of squares that contain a vertical line
%
\begin{example}\label{ex:line}
Let $r\geq 2$. Suppose $\Omega_{\text{line}}$ consists of an arbitrary number of disjoint squares that all intersect a given vertical line, that is there exists a frequency $\xi_0$ such that for every $\omega \in \Omega_{\text{line}}$ we have $\xi_0 \in \omega_1$. Observe that the frequency intervals $\omega_2$ must be all disjoint. We can bound pointwise 
\[ T_{\Omega_{\text{line}}}^r(f,g)(x) \leq \Big(\sum_{\omega\in\Omega_{\text{line}}} |\pi_{\omega_2}g(x)|^r \Big)^{1/r}\cdot \sup_{\omega\in\Omega_{\text{line}}}|\pi_{\omega_1} f(x)| \leq  \mathrm{RdF}^{r}(g)(x) \cdot \mathscr{C}f(x), \]
where $\mathscr{C}$ denotes the Carleson operator, which is bounded on $L^p$ for all $1<p<\infty$ (by the Carleson-Hunt theorem, \cite{Carleson}, \cite{Hunt}), and therefore we get that for $p>1$ and $q>r'$ (or $q \geq 2$ if $r\geq 2$) we can estimate for this particular collection
\[ \|T_{\Omega_{\text{line}}}^r(f,g)(x)\|_{L^s}\lesssim_{p,q} \|f\|_{L^p}\|g\|_{L^q}, \]
or in other words $C_{p,q,\Omega_{\text{line}}} \lesssim_{p,q} 1$ in the stated range.
\end{example}
%
%
% Example of a regular grid that shows you only get constant N^{1/2} by simple arguments
%
\begin{example}\label{ex:grid}
Let $r>2$ be fixed and consider now a collection of $N^{1/2} \times N^{1/2}$ points in $\widehat{\R^2}$ arranged in a rectangular grid with large spacing, and suppose that each point labeled by $(i,j)$ is the center of a square $\omega^{ij}$ and furthermore that the squares are all disjoint (their sidelengths can be all distinct). We let $\Omega_{\text{grid}} := \{\omega^{ij}\}_{i,j \leq N^{1/2}}$ and we try to bound $T_{\Omega_{\text{grid}}}^r$ in some range. Observe that since a priori
\[ |\pi_{\omega}(f,g)(x)| \leq \mathscr{C}f(x) \cdot \mathscr{C}g(x) \]
we always have the trivial bound 
\[ \|T_{\Omega_{\text{grid}}}^r(f,g)\|_{L^s} \lesssim_{p,q} N^{1/r} \|f\|_{L^p}\|g\|_{L^q} \]
for $p,q>1$. We can beat this trivial bound of $C_{p,q,\Omega_{\text{grid}}} \lesssim_{p,q} N^{1/r}$ by the following argument: since for a fixed $i$ the squares $\omega^{ij}$ are such that $\omega^{ij}_1$ all contain a same frequency as in the example above, we can bound pointwise
\begin{align*}
\Big(\sum_{\omega\in\Omega_{\text{grid}}} |\pi_{\omega}(f,g)(x)|^r \Big)^{1/r} = & \Big(\sum_{i \leq N^{1/2}} \sum_{j \leq N^{1/2}} |\pi_{\omega^{ij}_1}f(x) \cdot \pi_{\omega^{ij}_2}g(x)|^r \Big)^{1/r} \\
\leq & \mathscr{C}f(x) \cdot \Big(\sum_{i \leq N^{1/2}} \sum_{j \leq N^{1/2}} |\pi_{\omega^{ij}_2}g(x)|^r \Big)^{1/r} \\
\leq & \mathscr{C}f(x) \cdot \Big(\sum_{i \leq N^{1/2}} (\VarC{r}g(x))^r \Big)^{1/r} \\
= & N^{1/{2r}}\mathscr{C}f(x) \VarC{r}g(x),
\end{align*}
where $\VarC{r}$ is the Variational Carleson operator
\[ \VarC{r}f(x) := \sup_{M} \sup_{\xi_{1} < \cdots < \xi_{M}} \Big(\sum_{j=1}^{M-1} |\pi_{[\xi_j, \xi_{j+1}]}f(x)|^r\Big)^{1/r}. \]
It is known from \cite{OberlinSeegerTaoThieleWright} that this operator is $L^p \rightarrow L^p$ bounded for $r' < p < \infty$ if $r>2$, as is the case, and therefore we get for the range $p>1, q>r'$ an improvement in the dependence of the constant on the cardinality of $\Omega$ (specifically, $C_{p,q,\Omega_{\text{grid}}} \lesssim_{p,q} (\#\Omega)^{1/{2r}}$ instead of $(\#\Omega)^{1/r}$).
\end{example}
%
% natural conjecture
%
It is natural to conjecture that for some range of exponents (possibly as large as $p,q > r'$, like in \cite{BeneaBernicot}) one should have $C_{p,q,\Omega}\lesssim_{p,q} 1$ for every admissible $\Omega$, or in other words that inequality \eqref{eqn:target_estimate} should hold with constant $C_{p,q,\Omega}$ independent of $\Omega$, and specifically independent of its cardinality $\#\Omega$. Indeed, this is achieved in the $p,q>r'$ range by product-like collections of rectangles, that is collections of the form $\Omega = \{I\times J \text{ s.t. } I \in \mathcal{I}, J \in \mathcal{J}\}$, where $\mathcal{I}, \mathcal{J}$ are collections of disjoint intervals; this can be readily seen by a factorization of the operator and an application of Rubio de Francia's theorem.
Simple pointwise arguments like the one given in Example \ref{ex:grid} are unlikely to give such a result. However, by combining similar observations with the time-frequency analysis of \cite{BeneaBernicot} and some further ideas from \cite{Benea_thesis}, \cite{BeneaMuscalu_preprint} (as for example the consideration in the time-frequency analysis of an exceptional subset built from non-local operators), we are able to confirm the conjecture in the range given by $r' < p,q < r$.\\
%
% STATEMENT OF THE MAIN THEOREM
%
More precisely, let $\mathcal{D}$ denote the collection of dyadic intervals, that is intervals of the form $[n 2^k, (n+1)2^k]$ for arbitrary $n,k \in \Z$. Then our main result can be stated as follows.
\begin{theorem} \label{mainthm}
Let $r > 2$ be fixed. Then for all $p,q,s$ such that 
\[ \frac{1}{p} + \frac{1}{q} = \frac{1}{s} \]
and 
\[ r' < p,q < r, \quad r'/2 < s < r/2 \]
it holds that for every arbitrary collection $\Omega \subset \mathcal{D} \times \mathcal{D}$ of disjoint (dyadic) squares in $\widehat{\R^2}$ the estimate
\begin{equation}\label{eqn:main_theorem_estimate}
\Big\| \Big(\sum_{\omega\in\Omega} |\pi_{\omega}(f,g)(x)|^r \Big)^{1/r} \Big\|_{L^s} \lesssim_{p,q,r} \|f\|_{L^p} \|g\|_{L^q},
\end{equation}  
holds true for every $f \in L^p, g \in L^q$.
\end{theorem}
%
% remark on dyadic squares vs generic squares
%
\begin{remark}
In Theorem \ref{thm:smooth_squares_theorem} (from \cite{BeneaBernicot}) above, the statement encompasses arbitrary non-dyadic squares; this is because of the flexibility provided by the smoothness of the $\chi_\omega$ functions. However, in the non-smooth case things are not as simple. One can replace the assumption that the squares are dyadic with a well-separation assumption: namely, Theorem \ref{mainthm} still holds if we assume that $\Omega$ is a finite collection of arbitrary squares such that $4\omega \cap 4\omega' = \emptyset$ whenever $\omega \neq \omega'$ (essentially because then each square is contained in a unique dyadic square of comparable size). In the linear case it's always possible to reduce to a well-separated case (see \cite{RubioDeFrancia}) by means of classical Littlewood-Paley theory, but in the bilinear case such tools are not currently available.
\end{remark}
%
% remark on r=2
%
\begin{remark}
The condition $r>2$ is a shortcoming we inherit from \cite{BeneaBernicot}. However, for the $r=2$ case, one can deduce from the above theorem and H\"{o}lder's inequality that $T^2_{\Omega}$ is $L^2 \times L^2 \rightarrow L^1$ bounded with constant at most $O_{\eps}(\#\Omega^\eps)$ for any $\eps>0$. Indeed, we can bound pointwise $T_{\Omega}^2 (f,g) \leq \#\Omega^{\eps} T_{\Omega}^r (f,g)$, where $1/2 = 1/r + \eps$, and conclude using Theorem \ref{mainthm} for exponents $p=q=2$.
\end{remark}
This result partially confirms the natural conjecture stated above. Observe the range of boundedness provided by Theorem \ref{mainthm} is smaller than the corresponding one in Theorem \ref{thm:smooth_squares_theorem} above. We explain the reason why in Remark \ref{remark:motivation_for_smaller_range}. Figure \ref{figure:range_p_q} in section \S\ref{section:proof_of_main_theorem} provides a graphical illustration of the range obtained in Theorem \ref{mainthm}.\\ 
\par
%
% outline of the paper
%
The proof of Theorem \ref{mainthm} is presented in section \S\ref{section:proof}, and is split into a number of steps. The result is obtained by interpolation between a boundedness result for $T^{\infty}_{\Omega}$ (a trivial consequence of the Carleson-Hunt theorem) and a partial boundedness result for $T^r_{\Omega}$ when $r$ is close to $2$. The latter is obtained by adapting the time-frequency methods of \cite{BeneaBernicot} to our setup, but using non-local operators to construct the exceptional set as in \cite{Benea_thesis}, \cite{BeneaMuscalu_preprint}. The necessary preliminaries are carried out in sections \S\ref{section:discretization} - \S\ref{section:general_estimate}. The proof is concluded in \S\ref{section:proof_of_main_theorem}, where the particular interpolation result we will use (Lemma \ref{lemma:interpolation}) is also presented. Finally, we present a simple application in \S\ref{section:application}.
%
% acknowledgements
%
\subsection*{Acknowledgements} Both authors are supported by ERC project FAnFArE no. 637510. The authors are very grateful to Cristina Benea for many useful comments and discussions, and in particular for having shared with us a preprint of \cite{BeneaMuscalu_preprint}. The authors would also like to thank the  reviewers for their help with improving the presentation of the article. 
%
% SECTION WITH THE PROOF
%
\section{Proof of Theorem \ref{mainthm}}\label{section:proof}
%
% first reduction
% 
We can reduce the problem by linearization of the $\ell^r$ norm and duality to the following: given $f\in L^p, g \in L^q, h \in L^{s'}$ define the trilinear form
\begin{equation}\label{eqn:trilinear_form_1} 
\Lambda_r(f,g,h) := \int_{\R} \sum_{\omega\in\Omega} \pi_{\omega_1} f(x) \pi_{\omega_2} g(x) h_{\omega}(x) \ud x, 
\end{equation}
where $h_{\omega}(x) := h(x) \epsilon_{\omega}(x)$ and $\{\eps_\omega (x)\}_{\omega\in\Omega}$ satisfies $\|\{\epsilon_{\omega}(x)\}_{\omega}\|_{\ell^{r'}}\leq 1$ for every $x\in\R$; then it suffices to prove that 
\[ |\Lambda_r(f,g,h)| \lesssim_{p,q,r} \|f\|_{L^p} \|g\|_{L^q} \|h\|_{L^{s'}} \]
uniformly in $\{\eps_\omega (x)\}_{\omega\in\Omega}$. 
%
% second reduction
%
Thus we can further reduce the problem to that of bounding the trilinear form
\begin{equation}\label{eqn:trilinear_form_2}
 \Lambda(f,g,\mathbf{h}) := \int_{\R} \sum_{\omega\in\Omega} \pi_{\omega_1} f(x) \pi_{\omega_2} g(x) h_{\omega}(x) \ud x, 
\end{equation}
where $\mathbf{h} = \{h_\omega\}_{\omega\in\Omega}$ is a generic element of $L^{s'}(\ell^{r'})$.
\subsection{Discretization of the trilinear form}\label{section:discretization}
%
% DISCRETIZATION PROCEDURE
%
We perform a standard discretization procedure on the trilinear form $\Lambda$, except this time, contrary to one's expectations, we will not resolve the singularities using a Whitney decomposition. We have (using Radon duality, with $\ud\sigma$ the induced Lebesgue measure on the plane $\xi_1 + \xi_2 + \xi_3 = 0$)
\begin{align*}
\Lambda(f,g,\mathbf{h}) & = \int_{\R} \sum_{\omega\in\Omega} f \ast \widecheck{\mathds{1}_{\omega_1}}(x) g\ast \widecheck{\mathds{1}_{\omega_2}}(x) h_{\omega}(x) \ud x \\
& = \sum_{\omega\in\Omega} \int_{\xi_1 + \xi_2 + \xi_3 = 0} \widehat{f}(\xi_1) \mathds{1}_{\omega_1}(\xi_1) \widehat{g}(\xi_2)\mathds{1}_{\omega_1}(\xi_2) \widehat{h_{\omega}}(\xi_3) \ud \sigma(\xi_1,\xi_2,\xi_3) \\
& = \sum_{\omega\in\Omega} \int_{\xi_1 + \xi_2 + \xi_3 = 0} \widehat{f}(\xi_1) \mathds{1}_{\omega_1}(\xi_1) \widehat{g}(\xi_2)\mathds{1}_{\omega_1}(\xi_2) \widehat{h_{\omega}}(\xi_3) \chi_{\omega_3}(\xi_3) \ud \sigma(\xi_1,\xi_2,\xi_3) \\
& = \sum_{\omega\in\Omega} \int_{\R} f \ast \widecheck{\mathds{1}_{\omega_1}}(x) g\ast \widecheck{\mathds{1}_{\omega_2}}(x) h_{\omega} \ast \widecheck{\chi_{\omega_3}}(x) \ud x,
\end{align*}
where we have denoted $\omega_3 := 2(-\omega_1-\omega_2)$ and $\chi_{\omega_3}$ is a smoothed out characteristic function, identically equal to $1$ on $-\omega_1-\omega_2$ and identically vanishing outside $\omega_3$. Now, although the kernels decay very slowly, the functions $f\ast \widecheck{\mathds{1}_{\omega_j}}$ are morally still roughly constant in modulus at scale $|\omega_j|^{-1} =: |\omega|^{-1}$, and therefore it makes sense to do the following changes of variable:
\begin{align*}
\sum_{\omega\in\Omega} & \int_{\R} f \ast \widecheck{\mathds{1}_{\omega_1}}(x) g\ast \widecheck{\mathds{1}_{\omega_2}}(x) h_{\omega} \ast \widecheck{\chi_{\omega_3}}(x) \ud x \\
& = \sum_{\omega\in\Omega} |\omega|^{-1} \int_{\R} f \ast \widecheck{\mathds{1}_{\omega_1}}(|\omega|^{-1}y) g\ast \widecheck{\mathds{1}_{\omega_2}}(|\omega|^{-1}y) h_{\omega} \ast \widecheck{\chi_{\omega_3}}(|\omega|^{-1}y) \ud y \\
& = \sum_{\omega\in\Omega} \sum_{n \in \Z} |\omega|^{-1} \int_{0}^{1} f \ast \widecheck{\mathds{1}_{\omega_1}}(|\omega|^{-1}(n+z)) g\ast \widecheck{\mathds{1}_{\omega_2}}(|\omega|^{-1}(n+z)) \\
& \qquad \qquad\qquad \qquad \qquad \qquad \cdot h_{\omega} \ast \widecheck{\chi_{\omega_3}}(|\omega|^{-1}(n+z)) \ud z.
\end{align*}
In classical time-frequency analysis one rewrites the above form as an average over $z$ of discrete sums of coefficients, each given by an inner product against suitably defined wavepackets associated to tiles in the time-frequency plane, and then proceeds to bound the discrete sums uniformly in $z$; the approach we will take however is different and will involve allowing only a single scale for each square $\omega$, roughly speaking - a choice reflected in our definition of tri-tiles given below. This will allow us to do a time-frequency analysis of the trilinear form $\Lambda$ free from wavepackets (although wavepackets are intrinsically present in some strong results that we will use off-the-shelf).
%
%
% Introducing the TILES
%
Define then the tri-tiles as follows:
\begin{definition}
A tri-tile $P$ is a triple of sets of the form
\[ P = (P_1, P_2, P_3) = (I \times \omega_1, I \times \omega_2, I \times \omega_3) \]
where $\omega = \omega_1 \times \omega_2 \in \Omega$, $\omega_3 = 2(-\omega_1-\omega_2)$ as before and $I$ is a dyadic interval of length $|\omega|^{-1}$ . Sets $P_j$ for $j=1,2,3$ are referred to as tiles. Given a tri-tile $P$ we denote by $I_P$ the interval $I$ above; we also denote by $\omega(P) = \omega_1(P)\times \omega_2(P)$ the frequency square associated to the tri-tile $P$, and similarly we write $\omega_3(P)$ for $\omega_3$. Finally, given a collection of tiles $\P$ we denote by $\Omega(\P)$ the collection of frequency squares on which $\P$ is supported, namely 
\[ \Omega(\P) := \{ \omega \in \Omega \text{ s.t. } \omega = \omega(P) \text{ for some } P \in \P \}. \]
\end{definition}
Using H\"{o}lder's inequality on each summand above, we have
\begin{align*}
\Big||\omega|^{-1} \int_{0}^{1} f \ast \widecheck{\mathds{1}_{\omega_1}}(|\omega|^{-1}(n+z)) & g\ast \widecheck{\mathds{1}_{\omega_2}}(|\omega|^{-1}(n+z)) h_{\omega} \ast \widecheck{\chi_{\omega_3}}(|\omega|^{-1}(n+z)) \ud z\Big| \\
& \leq \|f \ast \widecheck{\mathds{1}_{\omega_1}}\|_{L^2(I_P)} \|g \ast \widecheck{\mathds{1}_{\omega_2}}\|_{L^2(I_P)} \|h_{\omega} \ast \widecheck{\chi_{\omega_3}}\|_{L^{\infty}(I_P)},
\end{align*}
where the tri-tile $P$ is given by
\[ I_P = [|\omega|^{-1}n, |\omega|^{-1}(n+1)] \]  
and 
\[ P = (I_P \times \omega_1, I_P \times \omega_2, I_P \times \omega_3). \]
Now, fix a parameter $r_0$ such that $2 < r_0 < r$. This will be fixed throughout the rest of the paper. Going back to writing $\pi_\omega f$ for $f \ast \widecheck{\mathds{1}_{\omega}}$, we notice that again by H\"{o}lder's inequality we have 
\[ \| \pi_{\omega_1(P)} f\|_{L^2(I_P)} \leq |I_P|^{1/2} \Big(\fint_{I_P}|\pi_{\omega_1(P)}f|^{r_0} \Big)^{1/r_0}, \]
and similarly for $g$. This may seem arbitrary at this point, but will be useful later in avoiding logarithmic-type losses in our estimates (see Remark \ref{remark_explanation_r_0}). We introduce the shorthand notation 
\begin{align*}
f(P_1) &:= \Big(\fint_{I_P}|\pi_{\omega_1(P)}f|^{r_0} \Big)^{1/r_0}, \\
g(P_2) &:= \Big(\fint_{I_P}|\pi_{\omega_2(P)}g|^{r_0} \Big)^{1/r_0},
\\
\mathbf{h}(P_3)& := \|h_{\omega(P)} \ast \widecheck{\chi_{\omega_3}}\|_{L^{\infty}(I_P)}.
\end{align*}
We have therefore shown that if $\P$ denotes the collection of all possible tri-tiles (obtained by letting $\omega$ range in $\Omega$ and $n \in \Z$, in the above notation), the trilinear form $\Lambda$ is bounded by the discretized sum
\[ |\Lambda(f,g,\mathbf{h})| \leq \Lambda_{\P}(f,g,\mathbf{h}) := \sum_{P \in \P} f(P_1) g(P_2) \mathbf{h}(P_3)|I_P|.  \]
The reason for this unusual choice of coefficients will become clear later in light of Lemma \ref{extra_term_control} and Proposition \ref{energy_1_2_control} below (see particularly Remarks \ref{remark_explanation_choice_of_coefficients} and \ref{remark_explanation_r_0}). In the rest of the section we will concentrate on bounding the discretized sum.
\subsection{Columns and rows} 
%
% INTRODUCING THE COLUMNS AND ROWS
%
We introduce here some structured collections of tri-tiles, originating from \cite{BeneaBernicot}, that will be fundamental to our analysis of the trilinear form $\Lambda$. They are to be thought of as the analogue for our setup of trees, in the language of classical time-frequency analysis.
\begin{definition}
Let $j = 1$ or $2$. Given tri-tiles $P, P'$ we say that $P <_j P'$ if 
\[ \omega_j(P) \supseteq \omega_j(P') \]
and 
\[ I_P \subseteq I_{P'}. \]
We say that $P \leq_j P'$ if $P <_j P'$ or $P = P'$.
\end{definition}
\begin{definition}
A collection of tri-tiles $\mathcal{C}$ is a \emph{column} if there exists a tri-tile $T \in \mathcal{C}$, referred to as the top of $\mathcal{C}$, such that for every $P \in \mathcal{C}$ it holds that $P \leq_1 T$.\\
Analogously, a collection of tri-tiles $\mathcal{R}$ is a \emph{row} if there exists a tri-tile $T \in \mathcal{R}$, referred to as the top of $\mathcal{R}$, such that for every $P \in \mathcal{R}$ it holds that $P \leq_2 T$.
\end{definition}
Given a column or row $\mathcal{T}$ we will use $\mathrm{Top}(\mathcal{T})$ to denote its top.
%
% remark on the structure of a column
%
\begin{remark} \label{remark_column_structure}
Observe that if $\mathcal{C}$ is a column then the collection of tiles $\{P_1 \text{ s.t. } P \in\mathcal{C}\}$ is overlapping, while the tiles $P_2$ for $P \in\mathcal{C}$ are all disjoint, and in particular if $\omega(P) \neq \omega(P')$ then $P_2$ and $P'_2$ are disjoint in frequency. The reverse holds for a row. This will be important later on.
\end{remark}
%
% COLUM/ROW ESTIMATE 
%
We show below that when $\mathcal{C}$ is a column we can give a good bound on $\Lambda_{\mathcal{C}}$ (and similarly for rows). In particular, we argue similarly to \cite{BeneaBernicot} and bound the discretized sum restricted to $\mathcal{C}$ as follows:
\begin{align*}
\sum_{P \in \mathcal{C}}f(P_1) g(P_2) \mathbf{h}(P_3)|I_P| \leq & \big(\sup_{P\in \mathcal{C}}f(P_1)\big) \sum_{P\in\mathcal{C}} g(P_2) \mathbf{h}(P_3)|I_P|\\
\leq & \big(\sup_{P\in \mathcal{C}}f(P_1)\big) \Big(\sum_{P \in \mathcal{C}} g(P_2)^r |I_P|\Big)^{1/r} \\
& \times \Big(\sum_{P \in\mathcal{C}} \mathbf{h}(P_3)^{r'} |I_P| \Big)^{1/{r'}}.
\end{align*}
Then, for the term in $g$ we bound 
\begin{align*}
\Big(\sum_{P \in \mathcal{C}} g(P_2)^r |I_P|\Big)^{1/r} & = \Big(\sum_{P \in \mathcal{C}} g(P_2)^{r - r_0} \cdot g(P_2)^{r_0} |I_P| \Big)^{1/r} \\
& \leq \big(\sup_{P \in \mathcal{C}} g(P_2)\big)^{(r-r_0)/r} \Big(\sum_{P \in \mathcal{C}} |I_P| g(P_2)^{r_0} \Big)^{1/r} \\
& = \big(\sup_{P \in \mathcal{C}} g(P_2)\big)^{(r-r_0)/r} \Big(\sum_{P \in \mathcal{C}} \int_{I_P} |\pi_{\omega_2(P)} g|^{r_0} \Big)^{1/r}
\end{align*}
(notice we have introduced the same type of quantity that controls the contribution of $f$ in here). As for the term in $\mathbf{h}$, we observe that 
\begin{align*}
|I_P| \mathbf{h}(P_3)^{r'} & = |I_P| \sup_{y \in I_P} |h_{\omega}\ast \widecheck{\chi_{\omega_3}}(y)|^{r'} \leq  |I_P|\big(\sup_{y \in I_P} \int |h_{\omega}(z)| |\widecheck{\chi_{\omega_3}}(y-z)| \ud z  \big)^{r'} \\
& \lesssim |I_P| \Big(\sup_{y \in I_P} \int |h_{\omega}(z)| \big(1 + \frac{|y-z|}{|I_P|}\big)^{-M} \frac{\ud z}{|I_P|}  \Big)^{r'} \\
& \leq |I_P| \big(\int |h_{\omega}(z)| \sup_{y \in I_P} \big(1 + \frac{|y-z|}{|I_P|}\big)^{-M} \frac{\ud z}{|I_P|}  \big)^{r'} \\
& \lesssim \int |h_{\omega}(z)|^{r'} \Phi_{I_P}(z) \ud z, 
\end{align*}
where $M>0$ is a large number and $\Phi_{I}$ denotes some rapidly decaying function concentrated in the interval $I$. Now observe that for each fixed $\omega$ the tiles $P$ which have $\omega$ as their frequency support have space support of fixed size $|I_P| = |\omega|^{-1}$, hence the intervals $I_P$ are all disjoint. Define then for an interval $I$ of length greater or equal to $|\omega|^{-1}$ the function
\[ \Phi^\omega_I (x) := \sum_{\substack{J \text{ dyadic s.t.} J\subseteq I, \\|J||\omega|=1}}\Phi_J(x); \]
notice $\Phi_{I}^{\omega}$ is essentially $\sim 1$ inside $I$ and decays like $(1 + |\omega|\mathrm{dist}(I,x))^{-M+1}$ outside of it (see remark \ref{remark:motivation_for_Phi_I_omega} for why we need to introduce such functions). We can thus bound
\begin{align*}
\sum_{P \in\mathcal{C}} \mathbf{h}(P_3)^{r'} |I_P| & \lesssim \sum_{P \in \mathcal{C}} \int |h_{\omega(P)}(z)|^{r'}\Phi_{I_P}(z)\ud z \\
& \lesssim \sum_{\omega} \int |h_{\omega}|^{r'} \Phi_{I_{\mathcal{C}}}^{\omega} \ud z.
\end{align*}
%
% INTRODUCING SIZES
%
To summarize, we introduce sizes:
\begin{definition}[Sizes]
For any collection of tri-tiles $\P$ define 
\begin{align*}
\size^1_f(\P)& := \sup_{P \in\P} f(P_1) = \sup_{P \in\P} \Big(\fint_{I_P} |\pi_{\omega_1(P)} f|^{r_0} \ud x\Big)^{1/r_0}, \\
\size^2_g(\P)& := \sup_{P \in\P} g(P_2) = \sup_{P \in\P} \Big(\fint_{I_P} |\pi_{\omega_2(P)} g|^{r_0} \ud x\Big)^{1/{r_0}}, \\
\size^3_{\mathbf{h}}(\P)& := \sup_{\mathcal{T} \subset \P} \Big( \frac{1}{|I_{\mathcal{T}}|} \sum_{\omega \in \mathcal{T}} \int |h_{\omega}|^{r'} \Phi_{I_{\mathcal{T}}}^{\omega} \ud z \Big)^{1/{r'}}, \\
\end{align*}
where the last supremum is taken over sub-collections $\mathcal{T}$ of $\P$ which are either rows or columns. 
\end{definition}
With this notation, what has been shown in this section can be summarized as 
\begin{proposition}\label{column_row_estimate}
Let $\mathcal{C}$ be a column of tri-tiles, then
\[ |\Lambda_{\mathcal{C}}(f,g,\mathbf{h})| \lesssim \Big(\frac{1}{|I_{\mathcal{C}}|}\sum_{P \in \mathcal{C}} \int_{I_P} |\pi_{\omega_2(P)} g|^{r_0} \Big)^{1/r} \size^1_f(\mathcal{C}) \big[\size^2_g(\mathcal{C})\big]^{(r-r_0)/r} \size^3_{\mathbf{h}}(\mathcal{C})  |I_\mathcal{C}|, \]
and similarly, if $\mathcal{R}$ is a row of tri-tiles, we have 
\[ |\Lambda_{\mathcal{R}}(f,g,\mathbf{h})| \lesssim  \Big(\frac{1}{|I_{\mathcal{C}}|}\sum_{P \in \mathcal{R}} \int_{I_P} |\pi_{\omega_1(P)} f|^{r_0} \Big)^{1/r} \big[\size^1_f(\mathcal{R})\big]^{(r-r_0)/r} \size^2_g(\mathcal{R}) \size^3_{\mathbf{h}}(\mathcal{R}) |I_\mathcal{R}|. \]
\end{proposition}
\subsection{Size bounds}
%
% CONTROLLING THE SIZES WITH AVERAGES
%
We have the following immediate bounds for the sizes introduced above:
%
% Controlling Size^1,2 with avg of Carleson operator
%
\begin{proposition}\label{size_1_2_control}
Let $\P$ be a collection of tri-tiles, then 
\[ \size^1_f(\P) \lesssim \sup_{P \in \P} \Big(\fint_{I_P} |\mathscr{C}f|^{r_0} \ud x \Big)^{1/r_0}, \]
where $\mathscr{C}$ is the Carleson operator. The analogous inequality holds for $\size^2_g$ as well.
\end{proposition}
\begin{proof}
Obvious.
\end{proof}
We do not state an analogous proposition for $\size^3$ since this size is already in a convenient form.\\
Later on we will also need the following simple bound in terms of $\VarC{r_0}$, the Variational Carleson operator as defined in Example \ref{ex:grid}.
%
% controlling the extra factor by Variational Carleson operator
%
\begin{lemma}\label{extra_term_control}
Let $\mathcal{C}$ be a column of tri-tiles. Then
\[ \frac{1}{|I_{\mathcal{C}}|}\sum_{P \in \mathcal{C}} \int_{I_P} |\pi_{\omega_2(P)} g|^{r_0} \ud x \leq \fint_{I_{\mathcal{C}}}|\VarC{r_0}g|^{r_0} \ud x. \]
\end{lemma}
Clearly, an analogous statement holds for rows.
\begin{proof}
If we rewrite the expression on the left hand side as 
\[ \fint_{I_{\mathcal{C}}} \sum_{P\in\mathcal{C}}|\pi_{\omega_2(P)}g(x)|^{r_0} \mathds{1}_{I_P}(x) \ud x, \]
then it follows from the definition of column (see Remark \ref{remark_column_structure}) that we can bound pointwise
\[ \Big(\sum_{P\in\mathcal{C}}|\pi_{\omega_2(P)}g(x)|^{r_0} \mathds{1}_{I_P}(x)\Big)^{r_0} \leq \VarC{r_0} g, \]
and the lemma follows.
\end{proof}
\subsection{Energies and energy estimates}
%
% INTRODUCING ENERGIES 
%
In this subsection we introduce the energies that will allow us to run a time-frequency argument for the trilinear form $\Lambda$. We prelude a definition of disjointness (taken from \cite{BeneaBernicot}) for collections of columns and collections of rows that is needed to state the definition of energies.
\begin{definition}
Given a collection $\mathfrak{C}$ of columns, we say that the columns in $\mathfrak{C}$ are \emph{mutually disjoint} if they are disjoint as sets of tri-tiles and if the sets $\mathsf{Top}(\mathcal{C})_1$ are disjoint (in the time-frequency plane) as $\mathcal{C}$ ranges over $\mathfrak{C}$. Equivalently, the tops are not comparable under $<_1$.\\
Analogously, given a collection $\mathfrak{R}$ of rows, we say that the rows in $\mathfrak{R}$ are mutually disjoint if they are disjoint as sets of tri-tiles and if the tiles $\mathsf{Top}(\mathcal{R})_2$ are disjoint (in the time-frequency plane) as $\mathcal{R}$ ranges over $\mathfrak{R}$. Equivalently, the tops are not comparable under $<_2$.
\end{definition}
Now we can state the definition of Energies.
% definition of energies
\begin{definition}[Energies]
% Energy^1,2
We denote
\[ \energy^1_f(\P) := \sup_{n \in \Z} \sup_{\mathfrak{C}} 2^n \Big(\sum_{\mathcal{C} \in \mathfrak{C}}  |I_{\mathsf{Top}(\mathcal{C})}|\Big)^{1/r_0},  \]
where the inner supremum runs over the collections $\mathfrak{C}$ of mutually disjoint columns in $\P$ such that for any column $\mathcal{C} \in \mathfrak{C}$ and for any $P \in \mathcal{C}$ it is 
\[ f(P_1) \geq 2^n. \]
Define analogously $\energy^2_g(\P)$ with respect to rows of tri-tiles in the obvious way.\\
% Energy^3
Finally, we denote 
\[ \energy^3_{\mathbf{h}}(\P) := \sup_{n \in \Z} \sup_{\mathfrak{T}} 2^n \Big(\sum_{\mathcal{T} \in \mathfrak{T}}  |I_{\mathsf{Top}(\mathcal{T})}|\Big)^{1/{r'}}, \]
where the inner supremum runs over the collections $\mathfrak{T}$ of mutually disjoint rows and columns in $\P$ such that for every $\mathcal{T} \in \mathfrak{T}$ it is
\[ \Big(\frac{1}{|I_{\mathsf{Top}(\mathcal{T})}|} \int \sum_{\omega \in \mathcal{T}} |h_{\omega}|^{r'} \Phi_{I_{\mathsf{Top}(\mathcal{T})}}^{\omega} \ud x \Big)^{1/{r'}} \geq 2^n. \]
\end{definition}
\begin{remark}
Notice that we are using $L^{r_0}$-type energies for $f$ and $g$, instead of $L^2$-type energies as in \cite{BeneaBernicot}. However, this is not the only difference. Even if one were to let $r_0 = 2$ above, our definition of $\energy^{1}$ would still be slightly different from the corresponding one of \cite{BeneaBernicot} (in particular it's somewhat relaxed) because in our arguments we won't have to resort to Bessel-type inequalities.
\end{remark}
We must show that these quantities are well-behaved in order for the machinery of time-frequency analysis to work. In particular, we ought to show that the energies can be controlled in terms of $L^p$ norms of the functions. This is what we do next. First of all, we have the simple
%
% ENERGY CONTROL FOR Energy^3
%
\begin{proposition}\label{energy_3_control}
For any collection of tri-tiles $\P$ and for any $\mathbf{h} \in L^{r'}(\ell^{r'})$ we have 
\[ \energy^3_{\mathbf{h}}(\P) \lesssim \|\mathbf{h}\|_{L^{r'}(\ell^{r'})}. \]
\end{proposition}
\begin{proof}
We may assume for simplicity that the collection $\P$ is finite, since our argument will not depend on its cardinality. Let $n \in \Z$ and $\mathfrak{T}$ be a collection of disjoint maximal rows and columns that realize the supremum in the definition of $\energy^3$, that is 
\[ \energy^3_{\mathbf{h}}(\P)^{r'} = 2^{r' n} \sum_{\mathcal{T} \in \mathfrak{T}} |I_{\mathsf{Top}(\mathcal{T})}|.  \] 
By definition we then have 
\begin{align*}
 \energy^3_{\mathbf{h}}(\P)^{r'} & \lesssim \sum_{\mathcal{T} \in \mathfrak{T}} \sum_{\omega \in \mathcal{T}} \int |h_\omega|^{r'} \Phi_{I_{\mathsf{Top}(\mathcal{T})}}^{\omega} \ud x \\ 
 & = \sum_{\omega \in \Omega} \int |h_\omega|^{r'}  \sum_{\substack{ \mathcal{T} \in \mathfrak{T} :\\ \mathcal{T} \ni \omega }}\Phi_{I_{\mathsf{Top}(\mathcal{T})}}^{\omega} \ud x;
 \end{align*}
but the collection $\mathfrak{T}$ is maximal with respect to inclusion and therefore for a fixed $\omega$ the intervals $I_{\mathsf{Top}(\mathcal{T})}$ are disjoint, hence by the rapid decay of the functions $\Phi^{\omega}$ we have 
\[ \sum_{\substack{ \mathcal{T} \in \mathfrak{T} :\\ \mathcal{T} \ni \omega }}\Phi_{I_{\mathsf{Top}(\mathcal{T})}}^{\omega} \lesssim 1, \]
and this concludes the proof.
\end{proof}
%
% remark on the need for \Phi_I^\omega to depend on \omega
%
\begin{remark}\label{remark:motivation_for_Phi_I_omega}
Two things should be noticed. Firstly, that in the last lines of the above proof we crucially needed the decay of the functions $\Phi_I^{\omega}$ away from $I$ to be controlled by $|\omega|^{-1}$ rather than by the larger $|I|$, which justifies their introduction and the subsequent definition of $\size^3$.\\
Secondly, the maximality we appealed to above means the following: given a column (or row) $\mathcal{C}$ such that 
\[ \frac{1}{|I_{\mathrm{Top}(\mathcal{C})}|}\int \sum_{\omega \in \Omega(\mathcal{C})} |h_\omega|^{r'} \Phi^\omega_{I_{\mathrm{Top}(\mathcal{C})}} \ud x \geq 2^{{r'}n}, \]
we can enlarge $\mathcal{C}$ by adjoining all tri-tiles $P \in \P$ such that $ \omega(P) \in \Omega(\mathcal{C})$ and $I_P \subseteq I_{\mathrm{Top}(\mathcal{C})}$, and doing so will not change the left hand side of the inequality at all. In other words, the only information $\size^3_{\mathbf{h}}$ is sensitive to is the space support of columns or rows $\mathcal{T}$ and the squares $\omega \in \Omega(\mathcal{T})$.
\end{remark}
Next we look at a bound for $\energy^{j}$ for $j=1,2$. 
%
% ENERGY CONTROL FOR Energy^1,2
%
\begin{proposition}\label{energy_1_2_control}
For any collection of tri-tiles $\P$ and for any $f\in L^{r_0}$ we have that 
\[ \energy^j_f(\P) \lesssim \|f\|_{L^{r_0}} \qquad j=1,2. \]
\end{proposition}
\begin{proof}
We let $j=1$ in here, the proof for $j=2$ being identical. Let $n \in \Z$ and $\mathfrak{C}$ be a collection of mutually disjoint columns that realize the supremum in the definition of $\energy^1_f(\P)$ within a factor of $2$, that is 
\[ \energy^1_f(\P)^{r_0} \sim 2^{r_0 n} \sum_{\mathcal{C} \in \mathfrak{C}} |I_{\mathsf{Top}(\mathcal{C})}|. \]
Then by definition of energy this implies 
\[ \energy^1_f(\P)^{r_0} \lesssim \sum_{\mathcal{C} \in \mathfrak{C}} \int_{I_{\mathcal{C}}} |\pi_{\omega_1(\mathcal{C})}f|^{r_0} \ud x, \]
where we have abused notation by writing $\mathcal{C}$ in place of $\mathsf{Top}(\mathcal{C})$ to ease readability. We rewrite the latter quantity as 
\[ \int \sum_{\mathcal{C} \in \mathfrak{C}} |\pi_{\omega_1(\mathcal{C})}f(x)|^{r_0} \mathds{1}_{I_{\mathcal{C}}}(x) \ud x. \]
Since by definition of $\energy^1$ the tiles $\mathsf{Top}(\mathcal{C})_1$ are disjoint in time-frequency, we can bound the sum pointwise (cf. the proof of Lemma \ref{extra_term_control}) in terms of the Variational Carleson operator, namely
\[ \sum_{\mathcal{C} \in \mathfrak{C}} |\pi_{\omega_1(\mathcal{C})}f(x)|^{r_0} \mathds{1}_{I_{\mathcal{C}}}(x) \leq \big(\VarC{r_0}f(x)\big)^{r_0}. \] 
We know from \cite{OberlinSeegerTaoThieleWright} that $\VarC{r_0}$ is $L^p \to L^p$ bounded for $p > r_0'$, and therefore (as $r_0 > 2$) our quantity above is bounded by 
\[ \int |\VarC{r_0}f|^{r_0} \ud x \lesssim \|f\|_{L^{r_0}}^{r_0}, \]
which finishes the proof.
\end{proof}
\begin{remark}\label{remark_explanation_choice_of_coefficients}
Lemma \ref{extra_term_control} and the above proposition are the reason for our choice of working with the coefficients $f(P_1),g(P_2), \mathbf{h}(P_3)$. Indeed, it is their form and precise localization (that is, the $L^{r_0}$ averages don't involve weights supported everywhere on $\R$) that allow us to introduce pointwise estimates of the relevant quantities in terms of Variational Carleson operators.
\end{remark}
\begin{remark}\label{remark_explanation_r_0}
The choice of introducing the parameter $r_0$ satisfying $2 < r_0 < r$ is motivated by Proposition \ref{energy_1_2_control} and Lemma \ref{extra_term_control}. Indeed, the problem here is that the Variational Carleson operator $\VarC{r_0}$ is only bounded when $r_0 > 2$. If we were to choose $r_0 = 2$ then, to argue by pointwise domination as above, we'd be forced to introduce a logarithmic-type loss (in the cardinality of $\Omega$) in our estimates, since the best one could say would then be that for any $\epsilon>0$
\[ \sum_{\mathcal{C} \in \mathfrak{C}} |\pi_{\omega_1(\mathcal{C})}f(x)|^{r_0} \mathds{1}_{I_{\mathcal{C}}}(x) \leq (\#\Omega)^{\epsilon} \big(\VarC{\frac{2}{1 - \epsilon}}f(x)\big)^2 \]
(by H\"{o}lder's inequality). The introduction of the parameter $r_0$ allows us to bypass this problem, at the cost of slightly reducing the range in which certain of our estimates hold (see Proposition \ref{prop:r=2}). This however won't be a problem in itself, in light of our use of interpolation in Section \S \ref{section:proof_of_main_theorem}.
\end{remark}
\subsection{Decomposition lemmas}
%
% CLASSICAL DECOMPOSITION LEMMAS
%
The decomposition lemma for $\size^1$ is well known and perhaps immediate. An identical result holds for $\size^2$ by replacing columns with rows.
%
% DECOMPOSITION LEMMA FOR Size^1,2
%
\begin{lemma}[Decomposition lemma for $\mathrm{Size}^1$]\label{decomposition_lemma_size_1_2}
Let $\P$ be a collection of tiles and let $n$ be such that 
\[ \size^1_{f}(\P) \leq 2^{-n} \energy^1_{f}(\P). \]
Then we can decompose $\P = \P_{\mathrm{low}} \sqcup \P_{\mathrm{high}}$ such that 
\[ \size^1_{f}(\P_{\mathrm{low}}) \leq 2^{-n-1} \energy^1_{f}(\P) \]
and $\P_{\mathrm{high}}$ can be organized into a collection $\mathfrak{C}$ of mutually disjoint columns $\mathcal{C}$ such that 
\[ \sum_{\mathcal{C}\in\mathfrak{C}} |I_{\mathcal{C}}| \lesssim 2^{r_0 n}.  \]
\end{lemma}
\begin{proof}
The proof is a simple stopping time argument. Let $\P_{\mathrm{stock}}$ be initialized to 
\[ \P_{\mathrm{stock}} := \{ P \in \P \text{ s.t. } f(P_1)> 2^{-n-1}\energy^1_{f}(\P) \}, \]
and let $\mathfrak{C}$ be initialized to $\mathfrak{C} := \emptyset$. We set right away 
\[ \P_{\mathrm{low}} := \P \backslash \P_{\mathrm{stock}}, \]
which will not be changed throughout the algorithm. The size property is then immediate from the definition of (the initial state of) $\P_{\mathrm{stock}}$. As for the organization of $\P_{\mathrm{high}} := \P \backslash \P_{\mathrm{low}}$ into columns, we proceed as follows. Of the tri-tiles in $\P_{\mathrm{stock}}$ that are maximal with respect to $<_1$, let $P_{\mathrm{max}}$ be the one such that $\inf I_{P_{\mathrm{max}}}$ is minimal and $\sup \omega_1(P_{\mathrm{max}})$ is maximal. Then we let $\mathcal{C}$ be the maximal (with respect to inclusion) column in $\P_{\mathrm{stock}}$ with top $P_{\mathrm{max}}$ and update $\mathfrak{C}$ to be $\mathfrak{C} \cup \{\mathcal{C}\}$ and update $\P_{\mathrm{stock}}$ to be $\P_{\mathrm{stock}}\backslash \bigcup_{P \in \mathcal{C}}\{P\}$. Repeat the process until $\P_{\mathrm{stock}}$ is empty and the algorithm stops. Then we see that by maximality the columns in $\mathfrak{C}$ are mutually disjoint, and as for the bound on $\sum_{\mathcal{C} \in \mathfrak{C}}|I_{\mathcal{C}}|$ notice that we have for each $\mathcal{C}$ and for each $P \in \mathcal{C}$
\[ f(P_1) > 2^{-n-1} \energy^1_{f}(\P)\]
and therefore just by definition of energy (and its monotonicity)
\[ 2^{-n-1} \energy^1_{f}(\P) \Big(\sum_{\mathcal{C}\in\mathfrak{C}} |I_{\mathcal{C}}|\Big)^{1/r_0} \lesssim \energy^1_{f}(\P_{\mathrm{high}}) \leq \energy^1_{f}(\P), \]
which proves the claim.  
\end{proof}
The decomposition lemma for $\mathrm{Size}^3$ is entirely similar (we have replaced the constant $2$ with $\gamma$ in view of its application to the proof of Lemma \ref{global_decomposition_lemma} below).
%
% DECOMPOSITION LEMMA FOR Size^3
%
\begin{lemma}[Decomposition lemma for $\mathrm{Size}^3$]\label{decomposition_lemma_size_3}
Let $\gamma = 2^{r_0/{r'}}$. Let $\P$ be a collection of tiles and let $n$ be such that 
\[ \mathrm{Size}^3_{h}(\P) \leq \gamma^{-n} \mathrm{Energy}^3_{h}(\P). \]
Then we can decompose $\P = \P_{\mathrm{low}} \sqcup \P_{\mathrm{high}}$ such that 
\[ \mathrm{Size}^3_{h}(\P_{\mathrm{low}}) \leq \gamma^{-n-1} \mathrm{Energy}^3_{h}(\P) \]
and $\P_{\mathrm{high}}$ can be organized into a collection $\mathfrak{T}$ of mutually disjoint columns and rows $\mathcal{T}$ such that 
\[ \sum_{\mathcal{T} \in\mathfrak{T}} |I_{\mathcal{T}}| \lesssim \gamma^{r'n} = 2^{r_0 n}.  \]
\end{lemma}
The proof is quite similar to the one given above for $\mathrm{Size}^1$, and is thus omitted.\\
Finally, by applying the decomposition lemmas simultaneously and then iterating one can achieve a global decomposition of a given collection $\P$ with good control of the sizes of the sub-collections. In particular
%
% GLOBAL DECOMPOSITION LEMMA
%
\begin{lemma}[global decomposition]\label{global_decomposition_lemma}
Let $\P$ be a collection of tri-tiles. Then there exists a partition $\P = \bigsqcup_{n} (\P_n^{\mathrm{col}} \sqcup \P_n^{\mathrm{row}})$ with the properties:
\begin{enumerate}[i)]
\item $\size^1_f(\P_n^{\mathrm{col},\mathrm{row}}) \lesssim \min(2^{-n} \energy^1_f(\P), \size^1_f(\P))$,
\item $\size^2_g(\P_n^{\mathrm{col},\mathrm{row}}) \lesssim \min(2^{-n} \energy^2_g(\P), \size^2_g(\P))$,
\item $\size^3_{\mathbf{h}}(\P_n^{\mathrm{col},\mathrm{row}}) \lesssim \min(2^{-r_0 n/{r'}} \energy^3_{\mathbf{h}}(\P), \size^3_{\mathbf{h}}(\P))$,
\item $\P_n^{\mathrm{col}}$ is organized into a collection $\mathfrak{C}_n$ of disjoint columns,
\item $\P_n^{\mathrm{row}}$ is organized into a collection $\mathfrak{R}_n$ of disjoint rows,
\item $\sum_{\mathcal{C}\in\mathfrak{C}_n}|I_{\mathcal{C}}| \lesssim 2^{r_0 n}$,
\item $\sum_{\mathcal{R}\in\mathfrak{R}_n}|I_{\mathcal{R}}| \lesssim 2^{r_0 n}$.
\end{enumerate}
The collection $\P_n^{\mathrm{col}}$ is empty if $n$ is such that 
\[ 2^{-n} \gtrsim \frac{\size^1_f(\P)}{\energy^1_f(\P)} \quad \text{and} \quad 2^{-r_0 n/{r'}} \gtrsim \frac{\size^3_{\mathbf{h}}(\P)}{\energy^3_{\mathbf{h}}(\P)}, \]
and similarly the collection $\P_n^{\mathrm{row}}$ is empty if $n$ is such that 
\[ 2^{-n} \gtrsim \frac{\size^2_g(\P)}{\energy^2_g(\P)} \quad \text{and} \quad 2^{-r_0 n/{r'}} \gtrsim \frac{\size^3_{\mathbf{h}}(\P)}{\energy^3_{\mathbf{h}}(\P)}. \]
\end{lemma}
\begin{proof}
Initialize $\P_{\mathrm{stock}} := \P$ and apply iteratively the decomposition Lemmas \ref{decomposition_lemma_size_1_2} and \ref{decomposition_lemma_size_3}, in the order given by whichever of the quantities 
\[\frac{\size^1_f(\P_{\mathrm{stock}})}{\energy^1_f(\P)}, \quad
\frac{\size^2_g(\P_{\mathrm{stock}})}{\energy^2_g(\P)}, \quad
\Big(\frac{\size^3_{\mathbf{h}}(\P_{\mathrm{stock}})}{\energy^3_{\mathbf{h}}(\P)}\Big)^{{r'}/r_0} \]
is largest, sorting columns and rows into the current $\P^{\mathrm{col}}, \P^{\mathrm{row}}$ respectively and updating $\P_{\mathrm{stock}}$ at the end of each step to be the collection $(\P_{\mathrm{stock}})_{\mathrm{low}}$ resulting from the last application of a decomposition lemma. We omit the details.
\end{proof}
\subsection{General estimate for $\Lambda_{\P}$}\label{section:general_estimate}
%
% GENERAL ESTIMATE FOR ARBITRARY COLLECTION OF TILES
%
In this section we prove the following general estimate, which will then be the main ingredient in the proof of Proposition \ref{prop:r=2} in \S \ref{section:proof_of_main_theorem}.
\begin{lemma}\label{general_estimate}
Let $\Omega, \Lambda$ be as above and let $\P$ be a collection of tri-tiles. Let $\sigma = \frac{r-r_0}{r}$ and denote for shortness 
\begin{align*}
&\size^1_f(\P) =: \mathcal{S}_1, \qquad \energy^1_f(\P)=:\mathcal{E}_1, \\
&\size^2_g(\P) =: \mathcal{S}_2, \qquad \energy^2_g(\P)=:\mathcal{E}_2, \\
&\size^3_{\mathbf{h}}(\P) =: \mathcal{S}_3, \qquad \energy^3_{\mathbf{h}}(\P)=:\mathcal{E}_3.
\end{align*}
Then we have 
\begin{equation}\label{eqn:general_estimate}
\begin{aligned}
|\Lambda_{\P}(f,g,\mathbf{h})| \lesssim & \Big[\sup_{P \in \P}\fint_{I_P} |\VarC{r_0}g|^{r_0}\Big]^{1/r} \\
& \qquad\qquad \times \mathcal{S}_1^{2\sigma\theta_1} \mathcal{E}_1^{1 - 2\sigma\theta_1} 
\mathcal{S}_2^{2\sigma\theta_2} \mathcal{E}_2^{\sigma - 2\sigma\theta_2} 
\mathcal{S}_3^{2\sigma\theta_3 {r'}/{r_0}} \mathcal{E}_3^{1 - 2\sigma \theta_3 {r'}/{r_0}}   \\
& +  \Big[\sup_{P \in \P}\fint_{I_P} |\VarC{r_0}f|^{r_0}\Big]^{1/r} \\
& \qquad\qquad \times \mathcal{S}_1^{2\sigma\xi_1} \mathcal{E}_1^{\sigma - 2\sigma\xi_1} 
\mathcal{S}_2^{2\sigma\xi_2} \mathcal{E}_2^{1 - 2\sigma\xi_2} 
\mathcal{S}_3^{2\sigma\xi_3 {r'}/{r_0}} \mathcal{E}_3^{1 - 2 \sigma\xi_3 {r'}/{r_0}} 
\end{aligned}
\end{equation}
for any $\theta_j, \xi_j$ such that $\theta_1 + \theta_2 + \theta_3 = 1$ and respectively $\xi_1 + \xi_2 + \xi_3 = 1$, and 
\begin{equation*}
\begin{aligned}[c]
0 & \leq \theta_1 \leq \min(1, (2\sigma)^{-1}),\\
0 & \leq \theta_2 \leq \frac{1}{2},\\
0 & < \theta_3 \leq 1,
\end{aligned}
\qquad
\begin{aligned}[c]
0 & \leq \xi_1 \leq \frac{1}{2},\\
0 & \leq \xi_2 \leq \min(1, (2\sigma)^{-1}),\\
0 & < \xi_3 \leq 1.
\end{aligned}
\end{equation*}
\end{lemma}
\begin{proof}
%
% use global decomposition lemma
%
Apply the global decomposition lemma (Lemma \ref{global_decomposition_lemma}) to the collection $\P$, thus obtaining a partition $\P = \bigsqcup_{n} \P_n^{\mathrm{col}} \sqcup \P_n^{\mathrm{row}}$. It suffices to consider the collections $\P_n^{\mathrm{col}}$ (which correspond to the first term in \eqref{eqn:general_estimate}), the proof for the collections $\P_n^{\mathrm{row}}$ being entirely analogous. Since $\P_n^{\mathrm{col}}$ is organized into a collection $\mathfrak{C}_n$ of disjoint columns, using Proposition \ref{column_row_estimate} we can bound 
\begin{align*}
 |\Lambda_{\P_n^{\mathrm{col}}}(f,g,\mathbf{h})| & \leq \sum_{\mathcal{C} \in \mathfrak{C}_n} |\Lambda_{\mathcal{C}}(f,g,\mathbf{h})| \\
 \lesssim & \sum_{\mathcal{C} \in \mathfrak{C}_n} \Big(\frac{1}{|I_{\mathcal{C}}|} \sum_{P\in\mathcal{C}} \int_{I_P}|\pi_{\omega_2(P)}g|^{r_0} \Big)^{1/r} \size^1_f(\mathcal{C}) \big[\size^2_g(\mathcal{C})\big]^{\sigma} \size^3_{\mathbf{h}}(\mathcal{C})  |I_\mathcal{C}| \\
 \lesssim & \sum_{\mathcal{C} \in \mathfrak{C}_n} \Big(\frac{1}{|I_{\mathcal{C}}|} \sum_{P\in\mathcal{C}} \int_{I_P}|\pi_{\omega_2(P)}g|^{r_0} \Big)^{1/r} \min(2^{-n}\mathcal{E}_1, \mathcal{S}_1) \big[\min(2^{-n}\mathcal{E}_2, \mathcal{S}_2)\big]^{\sigma} \\
 & \qquad\qquad\qquad\qquad\qquad\qquad \times \min(2^{-r_0 n/{r'}}\mathcal{E}_3, \mathcal{S}_3) |I_\mathcal{C}|; 
\end{align*} 
by Lemma \ref{extra_term_control} term $\Big(\frac{1}{|I_{\mathcal{C}}|} \sum_{P\in\mathcal{C}} \int_{I_P}|\pi_{\omega_2(P)}g|^{r_0} \Big)^{1/r}$ can be replaced with 
\[ \sup_{P \in \P}\Big[\fint_{I_P} |\VarC{r_0}g|^{r_0}\Big]^{1/r}, \]
which then factors out of the sum. By definition of $\P_n^{\mathrm{col}}$, what remains is controlled by 
\[ \min(2^{-n}\mathcal{E}_1, \mathcal{S}_1) \big[\min(2^{-n}\mathcal{E}_2, \mathcal{S}_2)\big]^{\sigma} \min(2^{-r_0 n/{r'}}\mathcal{E}_3, \mathcal{S}_3) 2^{r_0 n}, \]
and therefore it suffices to show that the sum over $n$ of all these contributions is controlled by the corresponding product of sizes and energies in the first term of the right hand side of \eqref{eqn:general_estimate}. This requires a tedious but easy case by case analysis.
%
% case by case analysis of the sum in n
%
Assume that 
\begin{equation}\label{eqn:order_assumption} 
\frac{\mathcal{S}_1}{\mathcal{E}_1} < \frac{\mathcal{S}_2}{\mathcal{E}_2} < \Big(\frac{\mathcal{S}_3}{\mathcal{E}_3}\Big)^{{r'}/{r_0}},
\end{equation}
the other cases being similar and thus omitted. We have
%
% cases according to the order of S_1/E_1 etc.
%
\begin{enumerate}[1)]
%
% case 1
%
\item \textbf{case} $2^{-n}\leq \frac{\mathcal{S}_1}{\mathcal{E}_1} < \frac{\mathcal{S}_2}{\mathcal{E}_2} < \Big(\frac{\mathcal{S}_3}{\mathcal{E}_3}\Big)^{{r'}/{r_0}}$: in this case the sum we have to bound becomes 
\begin{align*}
\sum_{n \;:\; 2^{-n} \leq \mathcal{S}_1 \mathcal{E}_1^{-1}} & 2^{-n} \mathcal{E}_1 2^{-n\sigma} \mathcal{E}_2^{\sigma} 2^{-r_0 n/{r'}}\mathcal{E}_3 2^{r_0 n}\\
= & \mathcal{E}_1 \mathcal{E}_2^{\sigma} \mathcal{E}_3 
\sum_{n \;:\; 2^{-n}\leq \mathcal{S}_1 \mathcal{E}_1^{-1}} 2^{-n(1 + \sigma - r_0/r)},
\end{align*}
and since $1 + \sigma - r_0 / r = 2\sigma$ the above evaluates to 
\[ \mathcal{S}_1^{2\sigma} \mathcal{E}_1^{1 - 2\sigma} \mathcal{E}_2^{\sigma} \mathcal{E}_3, \]
which by assumption \eqref{eqn:order_assumption} is clearly controlled by the desired 
\[ \mathcal{S}_1^{2\sigma\theta_1} \mathcal{E}_1^{1 - 2\sigma\theta_1} 
\mathcal{S}_2^{2\sigma\theta_2} \mathcal{E}_2^{\sigma - 2\sigma\theta_2} 
\mathcal{S}_3^{2\sigma\theta_3 {r'}/{r_0}} \mathcal{E}_1^{1 - 2\sigma\theta_3 {r'}/{r_0}}.\]
%
% case 2
%
\item \textbf{case} $\frac{\mathcal{S}_1}{\mathcal{E}_1} < 2^{-n} \leq \frac{\mathcal{S}_2}{\mathcal{E}_2} < \Big(\frac{\mathcal{S}_3}{\mathcal{E}_3}\Big)^{{r'}/{r_0}}$: in this case the sum becomes 
\begin{align*}
\sum_{n \;:\; \mathcal{S}_1 \mathcal{E}_1^{-1} < 2^{-n} \leq \mathcal{S}_2 \mathcal{E}_2^{-1}} & \mathcal{S}_1 2^{-n\sigma} \mathcal{E}_2^{\sigma} 2^{-r_0 n/{r'}}\mathcal{E}_3 2^{r_0 n}\\
= & \mathcal{S}_1 \mathcal{E}_2^{\sigma} \mathcal{E}_3 
\sum_{n \;:\; \mathcal{S}_1 \mathcal{E}_1^{-1} < 2^{-n} \leq \mathcal{S}_2 \mathcal{E}_2^{-1}} 2^{-n(\sigma - {r_0}/r)},
\end{align*}
and $\sigma - r_0 / r = 2\sigma -1$. Thus we have further sub-cases:
%
% subcases according to whether 2\sigma -1 >< 0 (that is, r><4)
%
\begin{enumerate}[i)]
% < 0
\item \textbf{subcase} $2\sigma -1 < 0$: in this case the sum is controlled by 
\[ \mathcal{S}_1 \mathcal{E}_2^{\sigma} \mathcal{E}_3 \Big(\frac{\mathcal{S}_1}{\mathcal{E}_1}\Big)^{2\sigma - 1} = \mathcal{S}_1^{2\sigma} \mathcal{E}_1^{1 - 2\sigma} \mathcal{E}_2^{\sigma} \mathcal{E}_3, \]
which we have already established is fine;
% > 0
\item \textbf{subcase} $2\sigma -1 > 0$: in this case the sum is controlled by 
\begin{align*}
\mathcal{S}_1 \mathcal{E}_2^{\sigma} \mathcal{E}_3 \Big(\frac{\mathcal{S}_2}{\mathcal{E}_2}\Big)^{2\sigma - 1} 
= \mathcal{E}_1 \Big(\frac{\mathcal{S}_1}{\mathcal{E}_1}\Big)^{2\sigma \theta_1} \mathcal{E}_2^{\sigma} \mathcal{E}_3 \Big(\frac{\mathcal{S}_1}{\mathcal{E}_1}\Big)^{1 - 2\sigma \theta_1}\Big(\frac{\mathcal{S}_2}{\mathcal{E}_2}\Big)^{2\sigma - 1},
\end{align*}
and since by assumption $1 - 2\sigma \theta_1 \geq 0$ we can further bound this by 
\begin{align*}
\mathcal{E}_1 \Big(\frac{\mathcal{S}_1}{\mathcal{E}_1}\Big)^{2\sigma \theta_1} \mathcal{E}_2^{\sigma} \mathcal{E}_3 \Big(\frac{\mathcal{S}_2}{\mathcal{E}_2}\Big)^{1 - 2\sigma \theta_1}\Big(\frac{\mathcal{S}_2}{\mathcal{E}_2}\Big)^{2\sigma - 1}
= \mathcal{E}_1 \Big(\frac{\mathcal{S}_1}{\mathcal{E}_1}\Big)^{2\sigma \theta_1} \mathcal{E}_2^{\sigma} \mathcal{E}_3 \Big(\frac{\mathcal{S}_2}{\mathcal{E}_2}\Big)^{2\sigma\theta_2 + 2\sigma \theta_3},
\end{align*}
which is clearly controlled by the desired quantity;
% = 0
\item \textbf{subcase} $2\sigma -1 = 0$: in this case the sum is controlled by 
\begin{align*}
\mathcal{S}_1 \mathcal{E}_2^{\sigma} \mathcal{E}_3 \log\Big(\frac{\mathcal{S}_2}{\mathcal{E}_2}\cdot \frac{\mathcal{E}_1}{\mathcal{S}_1}\Big) 
\lesssim & \mathcal{S}_1 \mathcal{E}_2^{\sigma} \mathcal{E}_3 \Big(\frac{\mathcal{S}_2}{\mathcal{E}_2}\cdot \frac{\mathcal{E}_1}{\mathcal{S}_1}\Big)^{2\sigma(\theta_2 + \theta_3)}\\
= & \mathcal{E}_1 \Big(\frac{\mathcal{S}_1}{\mathcal{E}_1}\Big)^{1 - 2\sigma(\theta_2 + \theta_3)} \mathcal{E}_2^{\sigma} \Big(\frac{\mathcal{S}_2}{\mathcal{E}_2}\Big)^{2\sigma(\theta_2 + \theta_3)} \mathcal{E}_3,
\end{align*}
which is again the desired quantity since for this value of $\sigma$ it is $1 - 2\sigma(\theta_2 + \theta_3) = 2\sigma\theta_1$.
\end{enumerate}
% end of subcases 
%
% case 3
%
\item \textbf{case} $\frac{\mathcal{S}_1}{\mathcal{E}_1} < \frac{\mathcal{S}_2}{\mathcal{E}_2} < 2^{-n} \leq\Big(\frac{\mathcal{S}_3}{\mathcal{E}_3}\Big)^{{r'}/{r_0}}$: in this case the sum becomes 
\begin{align*}
\sum_{n \;:\; \mathcal{S}_2 \mathcal{E}_2^{-1} < 2^{-n} \leq \mathcal{S}_3^{{r'}/{r_0}} \mathcal{E}_2^{-{r'}/{r_0}}} & \mathcal{S}_1 \mathcal{S}_2^{\sigma} 2^{-r_0 n/{r'}}\mathcal{E}_3 2^{r_0 n}\\
= & \mathcal{S}_1 \mathcal{S}_2^{\sigma} \mathcal{E}_3 
\sum_{n \;:\; \mathcal{S}_2 \mathcal{E}_2^{-1} < 2^{-n} \leq \mathcal{S}_3^{{r'}/{r_0}} \mathcal{E}_2^{-{r'}/{r_0}}} 2^{r_0 n/r};
\end{align*}
since $r_0/r = 1 - \sigma$, this is controlled by 
\[ \mathcal{S}_1 \mathcal{S}_2^{\sigma} \mathcal{E}_3 \Big(\frac{\mathcal{S}_2}{\mathcal{E}_2}\Big)^{\sigma - 1} = 
\mathcal{S}_1 \mathcal{E}_2^{\sigma} \mathcal{E}_3 \Big(\frac{\mathcal{S}_2}{\mathcal{E}_2}\Big)^{2\sigma - 1}, \]
which we have encountered in a previous case and is therefore fine too.
\end{enumerate}
Thus the proof is concluded.
\end{proof}
\subsection{Proof of the main theorem}\label{section:proof_of_main_theorem}
%
% PROOF FOR THE FULL OPERATOR
%
We are now ready to prove the main theorem (Theorem \ref{mainthm}). It will be obtained by interpolation between the two extreme situations, namely $r=\infty$ and $r$ close to $2$.\\
For the first case, we only use the Carleson operator which is bounded on all $L^p$ spaces for $p\in(1,\infty)$ to deduce the following:
%
% Proposition for r=\infty
%
\begin{proposition} \label{prop:r=infty} The bilinear operator $T^\infty_\Omega$ given by 
\[ T^\infty_\Omega (f,g)(x) := \sup_{\omega\in\Omega}|\pi_\omega(f,g)(x)| \]
is bounded from $L^p \times L^q$ to $L^s$ for all $1<p,q<\infty$, where $1/p + 1/q = 1/s$.
\end{proposition}
\begin{proof}
As observed in Example \ref{ex:grid}, the operator $T^\infty_\Omega$ is bounded pointwise by
\[ T^\infty_\Omega (f,g)(x) \leq \mathscr{C}f(x) \cdot \mathscr{C}g(x), \]
and the result then follows from the Carleson-Hunt theorem.
\end{proof}
%
%
% case r>2 but finite
%
For the second case, we will prove the following proposition, whose statement is identical to that of Theorem \ref{mainthm} except for the smaller range of $p,q$ (namely $p,q>r_0>2$ here, and hence $s>1$ too).
\begin{proposition} \label{prop:r=2}
Let $r> r_0 > 2$ be fixed (for interpolation purposes, $r$ should be thought of as being very close to $2$). Then for all $p,q,s$ such that 
\[ \frac{1}{p} + \frac{1}{q} = \frac{1}{s} \]
and 
\[ r_0 < p,q < r, \quad r_0/2 < s < r/2 \]
and for every arbitrary collection $\Omega$ of disjoint dyadic squares in $\widehat{\R^2}$, the estimate
\begin{equation}\label{eqn:prop_estimate}
\Big\| \Big(\sum_{\omega\in\Omega} |\pi_{\omega}(f,g)(x)|^r \Big)^{1/r} \Big\|_{L^s} \lesssim_{p,q,r,r_0} \|f\|_{L^p} \|g\|_{L^q},
\end{equation}  
holds true for every $f \in L^p, g \in L^q$.
\end{proposition}
\begin{remark}
Clearly, since the choice of $r_0 > 2$ is arbitrary, the above Proposition also holds all the way down to $r_0 = 2$. The definitions of Sizes and Energies have to change accordingly. However, for presentation purposes, we prefer to work with a fixed choice of Size and Energy since this doesn't affect the range one obtains after the interpolation argument that follows.
\end{remark}
Theorem \ref{mainthm} follows from Proposition \ref{prop:r=infty} and \ref{prop:r=2} by multilinear interpolation of vector-valued operators. More precisely, it will follow from a straightforward application of the next lemma (due to Silva, \cite{Silva}), which we state after a definition. 
%
% definition of generalized restricted weak type with vector valued operators
%
\begin{definition}
Let $\Lambda(f,g,\mathbf{h})$ be a trilinear form and let $(\alpha_1,\alpha_2,\alpha_3;t)$ be such that $0 \leq \alpha_1, \alpha_2 \leq 1$, $\alpha_3\leq 1$, $\alpha_1 + \alpha_2 + \alpha_3 = 1$ and $t\geq 1$. We say that $\Lambda$ is of \emph{generalized restricted weak type} $(\alpha_1,\alpha_2,\alpha_3;t)$ if for every measurable subsets $F,G,H$ of $\R$ of finite measure and all functions $f,g$ such that 
\[|f|\leq \mathds{1}_F, \quad |g|\leq \mathds{1}_G \]  
there exists a subset $H' \subseteq H$, called \emph{major subset}, such that $|H'| > \frac{1}{2} |H|$ and for all functions $\mathbf{h}$ such that 
\[ \Big(\sum_{k}|h_k|^t\Big)^{1/t} \leq \mathds{1}_{H'} \] 
the inequality 
\[ |\Lambda(f,g,\mathbf{h})| \lesssim |F|^{\alpha_1}|G|^{\alpha_2}|H|^{\alpha_3}\]
holds true.
\end{definition}
%
% MULTILINEAR INTERPOLATION OF VECTOR VALUED OPERATORS
%
\begin{lemma}[\cite{Silva}]\label{lemma:interpolation}
Let $\Lambda$ be a trilinear form of generalized restricted weak type $(\alpha_1,\alpha_2,\alpha_3;t_0)$ and $(\beta_1,\beta_2,\beta_3;t_1)$, with the property that the major subset doesn't depend on $(\alpha_1,\alpha_2,\alpha_3;t_0)$ or $(\beta_1,\beta_2,\beta_3;t_1)$. Then for all $\theta$ such that $0 < \theta < 1$, with 
\[ \alpha^{\theta}_j = (1-\theta)\alpha_j + \theta \beta_j, \qquad j=1,2,3  \]
and 
\[ \frac{1}{t_\theta} = \frac{1-\theta}{t_0} + \frac{\theta}{t_1}, \]
it holds that $\Lambda$ is of generalized restricted weak type $(\alpha^{\theta}_1,\alpha^{\theta}_2,\alpha^{\theta}_3, t_\theta)$.
\end{lemma}
\begin{proof}
The lemma is a particular case of a more general interpolation lemma originating from \cite{Silva} (specifically Lemma 4.3). We sketch the proof here for the reader's convenience.\\
We argue by complex interpolation. Let $F,G,H,H',f,g,\theta$ be given and let $\mathbf{h}$ be such that 
\[ \Big(\sum_k |h_k|^{t_\theta}\Big)^{1/{t_\theta}} \leq \mathds{1}_{H'}.\] 
For $z\in \mathbb{C}$ with $\mathrm{Re} z \in [0,1]$ define $\mathbf{h}^z$ by 
\[ h_k^z(x) := |h_k(x)|^{t(z)} \]
for every $k$, where 
\[ t(z) = (1-z)\frac{t_\theta}{t_0} + z \frac{t_\theta}{t_1}. \]
When $\mathrm{Re} z = 0$ we have $|h_k^z|^{t_0} = |h_k|^{t_\theta}$, and when $\mathrm{Re} z = 1$ we have $|h_k^z|^{t_1} = |h_k|^{t_\theta}$; hence by assumption we have for $\mathrm{Re} z = 0$
\[ |\Lambda(f,g,\mathbf{h}^z)| \lesssim |F|^{\alpha_1}|G|^{\alpha_2}|H|^{\alpha_3}, \]
and for $\mathrm{Re} z = 1$ we have
\[ |\Lambda(f,g,\mathbf{h}^z)| \lesssim |F|^{\beta_1}|G|^{\beta_2}|H|^{\beta_3}. \]
Since the function $\Phi(z) := \Lambda(f,g,\mathbf{h}^z)$ is easily seen to be holomorphic in the open strip $S = \{z \in \mathbb{C} \text{ s.t. } 0 < \mathrm{Re} z < 1 \}$, continuous in its closure and bounded, we can apply to it Hadamard's three-lines-lemma and conclude that since $\mathbf{h}^{\theta + i0} = \mathbf{h}$ we have
\[ |\Lambda(f,g,\mathbf{h})| \lesssim |F|^{\alpha^{\theta}_1}|G|^{\alpha^{\theta}_2}|H|^{\alpha^{\theta}_3}, \]
as desired.
\end{proof}
%
% showing how to apply the lemma to conclude the main theorem
%
Theorem \ref{mainthm} follows by taking $t_1 = \infty' = 1$ and $t_0$ sufficiently close to $2$ and applying Lemma \ref{lemma:interpolation} above to the trilinear form $\Lambda$ in \eqref{eqn:trilinear_form_2}. The hypotheses are verified by Propositions \ref{prop:r=infty} and \ref{prop:r=2} (the major subset for $t_1 = 1$ is just $H$ itself), and we thus get that for a given $r>2$ the trilinear form $\Lambda$ in \eqref{eqn:trilinear_form_2} is of generalized restricted weak type $(\alpha_1,\alpha_2,\alpha_3;r)$ for all $1/{r} < \alpha_1, \alpha_2 < 1/{r'}$; hence the trilinear form $\Lambda_r$ in \eqref{eqn:trilinear_form_1} is of generalized restricted weak type (in the classical sense) $(\alpha_1,\alpha_2,\alpha_3)$ for the same range of $\alpha$'s. Finally, the strong type estimates for $T_{\Omega}^r$ follow by classical multilinear interpolation.
\begin{remark}
The crucial point in the above reasoning is that for $r$ sufficiently close to $2$ Proposition \ref{prop:r=2} gives us a range of boundedness arbitrarily close to $L^2 \times L^2 \to L^1$.
\end{remark}

%
% GRAPHICAL REPRESENTATION OF THE RANGES OF BOUNDEDNESS
%
\begin{figure}[ht]
\centering
\begin{tikzpicture}[line cap=round,line join=round,>=triangle 45,x=1cm,y=1cm, scale=1.2] rectangle (7.13,4);
\fill[line width=0.5pt,dotted,fill=black,fill opacity=0.31] (1.4,1.4) -- (0.9,1.4) -- (0.9,0.9) -- (1.4,0.9) -- cycle;
\fill[line width=0.5pt,fill=black,fill opacity=0.1] (0.9,0.9) -- (0.9,2.1) -- (2.1,2.1) -- (2.1,0.9) -- cycle;
\draw [line width=0.5pt] (0,0) -- (0,3.2);
\draw [line width=0.5pt] (0,0)-- (3,0);
\draw [line width=0.5pt] (3,0)-- (3,3);
\draw [line width=0.5pt] (3,3)-- (0,3);
\draw [line width=0.5pt] (0,3)-- (0,0);
\draw [line width=0.5pt,dotted] (0,3)-- (3,0);
\draw [line width=0.5pt,dash pattern=on 1pt off 1pt] (0,1.4)-- (1.4,1.4);
\draw [line width=0.5pt,dash pattern=on 1pt off 1pt] (1.4,1.4)-- (1.4,0);
\draw [line width=0.5pt,dotted] (0.9,0.9)-- (0.9,2.1);
\draw [line width=0.5pt,dotted] (0.9,0.9)-- (2.1,0.9);
\draw [line width=0.5pt,dotted] (1.4,1.4)-- (0.9,1.4);
\draw [line width=0.5pt,dotted] (0.9,1.4)-- (0.9,0.9);
\draw [line width=0.5pt,dotted] (0.9,0.9)-- (1.4,0.9);
\draw [line width=0.5pt,dotted] (1.4,0.9)-- (1.4,1.4);
\draw [line width=0.5pt,dash pattern=on 1pt off 1pt] (0.9,0.9)-- (0.9,0);
\draw [line width=0.5pt,dash pattern=on 1pt off 1pt] (0.9,0.9)-- (0,0.9);
\draw [line width=0.5pt,dotted] (0.9,2.1)-- (2.1,0.9);
\draw [line width=0.5pt,dash pattern=on 1pt off 1pt] (0,2.1)-- (2.1,2.1);
\draw [line width=0.5pt,dash pattern=on 1pt off 1pt] (2.1,0)-- (2.1,2.1);
\draw (3.25,0.25) node[anchor=north west] {$1/p$};
\draw (-0.38,3.7) node[anchor=north west] {$1/q$};
\draw (-0.4,3.2) node[anchor=north west] {$1$};
\draw (2.85,0.0) node[anchor=north west] {$1$};
\draw [line width=0.5pt] (0,0) -- (3.191964136360795,0);
\draw (0.6,0) node[anchor=north west] {$1/r$};
\draw (1.1,0) node[anchor=north west] {$1/r_0$};
\draw (1.8,0) node[anchor=north west] {$1/{r'}$};
\draw (-0.6,1.15) node[anchor=north west] {$1/r$};\draw (-0.7,1.6) node[anchor=north west] {$1/r_0$};\draw (-0.65,2.35) node[anchor=north west] {$1/{r'}$};
\end{tikzpicture}
\caption{\footnotesize The darker square corresponds to the $p,q$ range given by Proposition \ref{prop:r=2}; interpolation with Proposition \ref{prop:r=infty} extends the range to that corresponding to the additional lighter area.} \label{figure:range_p_q}
\end{figure}
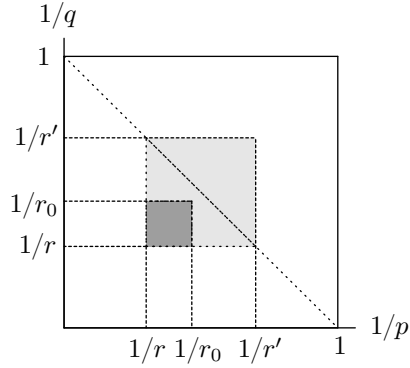
We end the proof of Theorem \ref{mainthm} by proving the last remaining proposition.
\begin{proof}[Proof of Proposition \ref{prop:r=2}]
%
% proof following the usual Muscalu, Tao, Thiele argument
%
The proof follows a standard argument originating from \cite{MuscaluTaoThiele} (although implicitly present in previous work). \\
% 
% reducing to multilinear interpolation
%
By multilinear interpolation, it suffices to prove restricted weak type estimates, that is it suffices to prove that if $F,G,H$ are measurable subsets of $\R$ of finite measure and $f,g$ are measurable functions such that 
\[ |f|\leq \mathds{1}_F, \quad |g|\leq \mathds{1}_{G}, \]
then there exists a subset $H'$ of $H$ such that $|H'| > 1/2 |H|$ and if
\[ \Big(\sum_{\omega\in\Omega}|h_\omega|^{r'}\Big)^{1/{r'}} \leq \mathds{1}_{H'} \]
then it holds that for any collection of tri-tiles $\P$ it is 
\begin{equation}\label{eqn:restricted_weak_type_estimate}
 |\Lambda_{\P}(f,g,\mathbf{h})|\lesssim_{p,q,r,r_0} |F|^{1/p} |G|^{1/q} |H|^{1/{s'}}.
 \end{equation}
%
% definition of Exceptional Set E
%
% (we have to introduce \mathfrak{p}, \mathfrak{q} to sort out the issue with the exponents in the general estimate - i.e. that without them, the exponents involving \theta would cancel each other out; this is essentially equivalent to forcing Size^1,2 to be O(1))
%
Given sets $F,G,H$ and functions $f,g$ as above, we fix two large numbers $\mathfrak{p}, \mathfrak{q}\gg r_0$, and we define the exceptional set $E$ to be 
\begin{align*}
E:= & \Big\{x \in \R \text{ s.t. } M(|\mathscr{C}f|^{\mathfrak{p}})(x) \gtrsim \frac{|F|}{|H|}\Big\} \\
& \cup \Big\{x \in \R \text{ s.t. } M(|\mathscr{C}g|^{\mathfrak{q}})(x) \gtrsim \frac{|G|}{|H|}\Big\} \\
& \cup \Big\{x \in \R \text{ s.t. } M(|\VarC{r_0}f|^{r_0})(x) \gtrsim \frac{|F|}{|H|}\Big\} \\
& \cup \Big\{x \in \R \text{ s.t. } M(|\VarC{r_0}g|^{r_0})(x) \gtrsim \frac{|G|}{|H|}\Big\},
\end{align*}
where $M$ is the dyadic Hardy-Littlewood maximal function. Define $H' := H \backslash E$; 
%
% proving the exceptional set is small compared to H
%
we claim that if we choose the implicit constants in the definition of $E$ to be large enough, we have $|H'|>  |H|/2$. Indeed, this follows from the $L^1 \rightarrow L^{1,\infty}$ boundedness of $M$ and the boundedness of the relevant operators for the given exponents, for example
\begin{align*}
\Big|\Big\{x \in \R \text{ s.t. } M(|\mathscr{C}f|^{\mathfrak{p}})(x) \gtrsim \frac{|F|}{|H|}\Big\}\Big| & \lesssim \frac{\|\mathscr{C}f^{\mathfrak{p}}\|_{L^1}}{|F|}|H| \\
& = \frac{\|\mathscr{C}f\|_{L^{\mathfrak{p}}}^{\mathfrak{p}}}{|F|}|H| \lesssim_{\mathfrak{p}} \frac{\|f\|_{L^{\mathfrak{p}}}^{\mathfrak{p}}}{|F|}|H| \ll |H|,
\end{align*}
where we have used the Carleson-Hunt theorem in the second to last inequality. The same holds for the other terms in the definition of $E$, where in particular one also has to invoke the $L^{r_0} \to L^{r_0}$ boundedness of $\VarC{r_0}$ proven in \cite{OberlinSeegerTaoThieleWright}.\\
%
% splitting the partition according to whether I_P is contained in the exceptional set or not
%
Now, we partition the collection $\P$ into
\begin{align*}
\P_{\mathrm{small}} :=& \{P \in \P \text{ s.t. } I_P \not\subset E\}, \\
\P_{\mathrm{large}}:=& \P \backslash \P_{\mathrm{small}},
\end{align*}
and will estimate separately the trilinear forms $\Lambda_{\P_{\mathrm{small}}}$ and $\Lambda_{\P_{\mathrm{large}}}$.\\
%
% proof for tiles that are not contained in E
%
% estimates for every term that appears in the general estimate
%
We start with $\P_{\mathrm{small}}$. Since $\mathfrak{p}>r_0$ we have 
\[ \Big(\fint_{I_P} |\mathscr{C}f|^{r_0} \Big)^{1/r_0} \leq \Big(\fint_{I_P} |\mathscr{C}f|^{\mathfrak{p}} \Big)^{1/\mathfrak{p}}, \]
so given $P \in \P_{\mathrm{small}}$ we observe that since $I_P \not\subset E$ we must have (see Proposition \ref{size_1_2_control})
\[ \size^1_f(\P_{\mathrm{small}}) \lesssim_{\mathfrak{p}} \Big(\frac{|F|}{|H|}\Big)^{1/\mathfrak{p}}. \]
Similarly, we see that 
\begin{align*}
\sup_{P \in \P_{\mathrm{small}}} \fint_{I_P}|\VarC{r_0}f|^{r_0} & \lesssim \frac{|F|}{|H|}, \\
\size^2_g(\P_{\mathrm{small}}) & \lesssim_{\mathfrak{q}} \Big(\frac{|G|}{|H|}\Big)^{1/\mathfrak{q}}, \\
\sup_{P \in \P_{\mathrm{small}}} \fint_{I_P}|\VarC{r_0}g|^{r_0} & \lesssim \frac{|G|}{|H|};
\end{align*}
moreover, we have trivially
\[ \size^3_{\mathbf{h}}(\P_{\mathrm{small}}) \lesssim 1. \]
%
% applying the estimates to the general estimate putting everything together
%
Combining this information with the general estimate in Lemma \ref{general_estimate} (for which we set $\theta_j = \xi_j$ for $j=1,2,3$, thus forcing the condition $0 \leq \theta_1, \xi_1, \theta_2, \xi_2 \leq 1/2$) and the energy estimates in Propositions  \ref{energy_3_control}, \ref{energy_1_2_control}, we obtain after some algebra (recall that $\sigma = (r - r_0)/r$)
\begin{equation}\label{eqn:main_estimate}
\begin{aligned}
|\Lambda_{\P_{\mathrm{small}}}(f,g,\mathbf{h})|  \lesssim_{\mathfrak{p},\mathfrak{q}} & \Big[\frac{|G|}{|H|}\Big]^{1/r} \Big(\frac{|F|}{|H|}\Big)^{2 \sigma \theta_1 / \mathfrak{p}} |F|^{(1 - \sigma \theta_1)/r_0} \\
& \qquad \times \Big(\frac{|G|}{|H|}\Big)^{2\sigma \theta_2/\mathfrak{q}} |G|^{(\sigma - 2\sigma\theta_2)/r_0} \cdot 1 \cdot |H|^{1/{r'} - 2\sigma \theta_3 / {r_0}} \\
& +  \Big[\frac{|F|}{|H|}\Big]^{1/r} \Big(\frac{|F|}{|H|}\Big)^{2 \sigma \theta_1 / \mathfrak{p}} |F|^{(\sigma - 2\sigma \theta_1)/r_0} \\
& \qquad \times \Big(\frac{|G|}{|H|}\Big)^{2\sigma \theta_2/\mathfrak{q}} |G|^{(1 - \sigma\theta_2)/r_0} \cdot 1 \cdot |H|^{1/{r'} - 2\sigma \theta_3 / {r_0}} \\
 = & |F|^{1/r_0  - 2\sigma(1/r_0 - 1/\mathfrak{p})\theta_1} |G|^{1/r_0 - 2\sigma(1/r_0 - 1/\mathfrak{q})\theta_2} \\
 & \qquad \times |H|^{1/{r'} - 1/r - 2\sigma \theta_3 / {r_0} - 2\sigma \theta_1/ \mathfrak{p} - 2 \sigma \theta_2 /\mathfrak{q} }.
\end{aligned}
\end{equation}
By choosing $\mathfrak{p}, \mathfrak{q}$ large enough, we obtain \eqref{eqn:restricted_weak_type_estimate} for any choice of exponents in the stated range. Notice that in order to prove Theorem \ref{mainthm} by interpolation we don't need the full range of exponents provided by Proposition \ref{prop:r=2}; it suffices to take $\mathfrak{p},\mathfrak{q}$ large but fixed for each $r$ (for example $r_0 = (r+2)/2$ and $\mathfrak{p}= \mathfrak{q} = 1000 r_0$), so that the hypotheses of the interpolation lemma \ref{lemma:interpolation} apply, to conclude Theorem \ref{mainthm}.\\
%
% proof for the tiles that are contained in E (essentially a repetition of the above)
%
Now we are left with showing that \eqref{eqn:main_estimate} holds for $\Lambda_{\P_{\mathrm{large}}}$ as well. In order to do so, we decompose $\P_{\mathrm{large}}$ into $\bigsqcup_{d \in \N} \P_{d}$ where
%
% decomposition according to the distance from the complement of E
%
\[ \P_d := \Big\{P \in \P_{\mathrm{large}} \text{ s.t. } 1 + \frac{\mathrm{dist}(I_P, E^c)}{|I_P|} \sim 2^d \Big\}; \]
it then suffices to prove that the contribution of $\Lambda_{\P_d}$ is summable in $d$ and the sum is bounded by \eqref{eqn:main_estimate}. 
%
% Estimates for every term that appears in the general estimate for \P_d
%
Let then $d$ be fixed and observe that if $P \in \P_d$ then $2^{d+O(1)}I_P \not\subset E$, thus as seen above we must have (up to finding a dyadic interval $I$ that contains $2^{d+O(1)}I_P$ and has comparable length, which is always possible)
\[ \fint_{2^{d+O(1)}I_P} |\mathscr{C}f|^{\mathfrak{p}} \lesssim \frac{|F|}{|H|}, \]
and therefore 
\[ \fint_{I_P} |\mathscr{C}f|^{\mathfrak{p}} \lesssim 2^{d} \frac{|F|}{|H|}, \]
and hence by Proposition \ref{size_1_2_control} and H\"{o}lder's inequality
\[ \size^1_f(\P_d) \lesssim_{\mathfrak{p}} 2^{d/\mathfrak{p}} \Big(\frac{|F|}{|H|}\Big)^{1/\mathfrak{p}}. \]
Similarly we have 
\begin{align*}
\sup_{P \in \P_{d}} \fint_{I_P}|\VarC{r_0}f|^{r_0} & \lesssim_\eps 2^d \frac{|F|}{|H|}, \\
\size^2_g(\P_{d}) & \lesssim_{\mathfrak{q}} 2^{d/\mathfrak{q}} \Big(\frac{|G|}{|H|}\Big)^{1/\mathfrak{q}}, \\
\sup_{P \in \P_{d}} \fint_{I_P}|\VarC{r_0}g|^{r_0} & \lesssim_\eps 2^d\frac{|G|}{|H|}.
\end{align*}
%
% better estimate for Size^3
%
However, for $\size^3_{\mathbf{h}}$ we now have a better estimate, namely for any $P \in \P_d$ it must be
\[ \frac{1}{|I_P|}\int \sum_{\omega\in\Omega} |h_\omega|^{r'} \Phi_{I_P}^{\omega} \lesssim 2^{-dM} \]
for a large $M>0$ of our choice, thanks to the fast decay of the functions $\Phi_{I_P}^{\omega}$ (here for convenience we are writing $\Phi_I^\omega$ for $\Phi_I$ even when $|I|< |\omega|^{-1}$), and this estimate in turn implies the bound 
\[ \size^3_{\mathbf{h}}(\P_d) \lesssim 2^{-dM/{r'}}. \]
%
% applying the estimates to the general estimate putting everything together
%
If we apply the general estimate of Proposition \ref{general_estimate} to $\P_d$ as done before we then get 
\begin{align*}
|\Lambda_{\P_d}(f,g,\mathbf{h})|\lesssim_{\mathfrak{p},\mathfrak{q}} 2^{-M'd}  & 
|F|^{1/r_0  - 2\sigma(1/r_0 - 1/\mathfrak{p})\theta_1}|G|^{1/r_0 - 2\sigma(1/r_0 - 1/\mathfrak{q})\theta_2} \\
 & \times |H|^{1/{r'} - 1/r - 2\sigma \theta_3 / {r_0} - 2\sigma \theta_1/ \mathfrak{p} - 2 \sigma \theta_2 /\mathfrak{q} }
\end{align*}
for some large $M'>0$ depending on $M, r$. As this is summable in $d$, the proof is concluded.
\end{proof}
%
% REMARK on why the range is more limited than in the smooth squares case 
%
\begin{remark}\label{remark:motivation_for_smaller_range}
We comment here on why, aside from the logarithmic loss, we cannot recover the same range for $\Lambda$ as in Theorem \ref{thm:smooth_squares_theorem} (\cite{BeneaBernicot}), in which the bilinear frequency projections onto the $\omega$'s are taken to be smooth. If one uses the appropriate version of the above argument in that context, the range obtained is symmetric with respect to $2$, that is one gets estimates for all $r' < p,q < r$ directly, without the need to appeal to interpolation results like Lemma \ref{lemma:interpolation}. The reason behind this is two-fold: firstly there's the fact that in that case all sizes satisfy $\size^{j}(\P)\lesssim 1$ a priori for $j=1,2,3$ (and with this information alone one already obtains the range $2<p,q< r$); and secondly the sizes are controlled by $L^1$-averages instead of $L^{r_0}$-averages as in our case (see Proposition \ref{size_1_2_control}). Thus in the smooth case of \cite{BeneaBernicot} one can effectively bound $\size^1_f(\P_{\mathrm{small}}) \lesssim \min(1, |F|/|H|)$ and similarly for $\size^2$, which then yields the wider range described above. Our use of the $\mathfrak{p},\mathfrak{q}$ powers essentially amounts to a substitute for the condition $\size^{j}(\P)\lesssim 1$, hence the smaller range.\\
Finally, the full range $p,q > r'$ in \cite{BeneaBernicot} is obtained by a further argument involving the localization of sizes and energies; alternatively, one can obtain it by considering the formal adjoints of the bilinear operator. In the non-smooth case both approaches fail: our sizes and energies don't localize well, since we are controlling them with non-local operators; and the formal adjoints cannot be simply reduced to the original operator, so that the analysis developed in here doesn't extend to them automatically.
\end{remark}
\section{Application to bilinear multipliers}\label{section:application}
Let $\Omega$ be a collection of dyadic frequency squares, not necessarily finite and not necessarily disjoint, and let $\mathbf{a}=\{a_{\omega}\}_{\omega \in \Omega}$ be a sequence of complex coefficients; form then the bilinear multiplier $T$ given by 
\[ T_{\mathbf{a}}(f,g)(x) := \sum_{\omega \in \Omega} a_\omega \pi_\omega(f,g)(x). \]
We are interested in finding conditions on $\Omega$ and $\{a_{\omega}\}_{\omega \in \Omega}$ which ensure the $L^p \times L^q \rightarrow L^s$ boundedness of $T$ in some range of exponents $p,q,s$.\\
%
% Situation where the Carleson Condition is satisfied for some power of the a_\omega
%
Consider the following situation: assume that for some $\beta \in (0,2)$ we have $\|\mathbf{a}\|_{\ell^\beta}<\infty$, and moreover the coefficients $a_\omega$ satisfy the Carleson Condition
\begin{equation}\label{Carleson_condition}
\sum_{\substack{\omega' \in \Omega, \\ \omega' \subset \omega}} |a_{\omega'}|^\beta \leq C |a_{\omega}|^\beta, \qquad \forall \omega \in \Omega.
\end{equation}
Then we argue that the bilinear multiplier $T_{\mathbf{a}}$ is bounded from $L^p \times L^q$ into $L^s$ with $1/p + 1/q = 1/s$ for $\beta < p,q < \beta'$, where $\beta'$ is replaced by $\infty$ if $\beta\leq 1$. Indeed, we partition the collection $\Omega$ as follows: let $n \in \N$ and define the sub-collection
\[ \Omega_n := \{ \omega \in \Omega \text{ s.t. } |a_\omega| \sim 2^{-n}\|\mathbf{a}\|_{\ell^\beta} \}; \]
then clearly 
\begin{equation}\label{eqn:cardinality_Omega_n}
\#\Omega_n \lesssim 2^{\beta n} 
\end{equation}
and moreover every collection $\Omega_n$ is the union of $O(1)$ collections of disjoint dyadic squares. This last fact is due to the Carleson Condition, since for every $\omega_0 \in \Omega_n$ it must be by definition
\[ C |a_{\omega_0}|^\beta \geq \sum_{\substack{\omega \in \Omega_n, \\ \omega \subset \omega_0}} |a_{\omega}|^\beta \sim |a_{\omega_0}|^\beta \#\{ \omega \in \Omega_n \text{ s.t. } \omega \subset \omega_0 \};  \]
thus if we do a generational decomposition of $\Omega_n$ (starting from the collection of maximal elements with respect to inclusion and so on), we will encounter at most $O(1)$ generations, which proves the claim.\\
Assume henceforth for the sake of clarity that for each $n$ the collection $\Omega_n$ consists of disjoint dyadic squares only. If we take $r \in (2,\beta')$ we can bound
\begin{align*}
\Big|\sum_{\omega \in \Omega} a_\omega \pi_\omega(f,g)(x)\Big| \leq & \sum_{n\in\N} \Big(\sum_{\omega \in \Omega_n} |a_{\omega}|^{r'}\Big)^{1/{r'}}\Big(\sum_{\omega \in \Omega_n} |\pi_\omega(f,g)(x)|^r\Big)^{1/r} \\
\sim & \sum_{n\in\N} 2^{-n} \|\mathbf{a}\|_{\ell^\beta} (\#\Omega_n)^{1/{r'}} T_{\Omega_n}^r(f,g)(x).
\end{align*}
By Theorem \ref{mainthm} and triangle inequality we then have that 
\[ \|T_{\mathbf{a}}(f,g)\|_{L^s} \lesssim_{p,q} \|\mathbf{a}\|_{\ell^\beta} \|f\|_{L^p}\|g\|_{L^q} \sum_{n \in \N} 2^{-n} (\#\Omega_n)^{1/{r'}}, \] 
but by \eqref{eqn:cardinality_Omega_n} the sum is bounded by 
\[ \sum_{n \in \N} 2^{-n} 2^{\beta n /{r'}} \lesssim 1, \]
thanks to our choice of $r$.
%
% remark on the Carleson condition
\begin{remark}
The Carleson condition \eqref{Carleson_condition} is introduced to enforce the fact that the collections $\Omega_n$ are made of essentially disjoint squares, and in particular they can be decomposed into at most $O(1)$ collections of disjoint squares. But actually, if we had that for some $\delta<1$ each $\Omega_n$ can be decomposed into at most $O(\#\Omega_n^\delta)$ collections of disjoint squares, we could still bound the multiplier in a (smaller) range.
\end{remark}

\bibliography{bilinear_rubio_de_francia_bibliography}

\providecommand{\bysame}{\leavevmode\hbox to3em{\hrulefill}\thinspace}
\providecommand{\MR}{\relax\ifhmode\unskip\space\fi MR }
% \MRhref is called by the amsart/book/proc definition of \MR.
\providecommand{\MRhref}[2]{%
  \href{http://www.ams.org/mathscinet-getitem?mr=#1}{#2}
}
\providecommand{\href}[2]{#2}
\begin{thebibliography}{10}

\bibitem{Benea_thesis}
Cristina Benea, \emph{Vector-valued {E}xtensions for {S}ingular {B}ilinear
  {O}perators and {A}pplications}, Ph.D. thesis, Cornell University,
  https://ecommons.cornell.edu/handle/1813/40903, 2015.

\bibitem{BeneaBernicot}
Cristina Benea and Fr\'ed\'eric Bernicot, \emph{A bilinear {R}ubio de {F}rancia
  inequality for arbitrary squares}, Forum Math. Sigma \textbf{4} (2016), e26,
  34. \MR{3569060}

\bibitem{BeneaMuscalu_preprint}
Cristina Benea and Camil Muscalu, \emph{Rubio de {F}rancia theorems revisited:
  linear and bilinear case}, preprint.

\bibitem{BeneaMuscalu}
\bysame, \emph{Multiple vector-valued inequalities via the helicoidal method},
  Anal. PDE \textbf{9} (2016), no.~8, 1931--1988. \MR{3599522}

\bibitem{Bernicot}
Fr\'ed\'eric Bernicot, \emph{{$L^p$} estimates for non-smooth bilinear
  {L}ittlewood-{P}aley square functions on {$\Bbb R$}}, Math. Ann. \textbf{351}
  (2011), no.~1, 1--49. \MR{2824844}

\bibitem{BernicotShrivastava}
Fr\'ed\'eric Bernicot and Saurabh Shrivastava, \emph{Boundedness of smooth
  bilinear square functions and applications to some bilinear
  pseudo-differential operators}, Indiana Univ. Math. J. \textbf{60} (2011),
  no.~1, 233--268. \MR{2952417}

\bibitem{Bourgain}
J.~Bourgain, \emph{On the behavior of the constant in the {L}ittlewood-{P}aley
  inequality}, Geometric aspects of functional analysis (1987--88), Lecture
  Notes in Math., vol. 1376, Springer, Berlin, 1989, pp.~202--208. \MR{1008724}

\bibitem{Carleson}
Lennart Carleson, \emph{On convergence and growth of partial sums of {F}ourier
  series}, Acta Math. \textbf{116} (1966), 135--157. \MR{0199631}

\bibitem{Carleson_LP}
\bysame, \emph{On the {L}ittlewood-{P}aley theorem}, Inst. Mittag-Leffler
  report (1967).

\bibitem{Cordoba}
Antonio C\'ordoba, \emph{Some remarks on the {L}ittlewood-{P}aley theory},
  Rend. Circ. Mat. Palermo (2) (1981), no.~suppl. 1, 75--80. \MR{639467}

\bibitem{TaoCowling}
Michael Cowling and Terence Tao, \emph{Some light on {L}ittlewood-{P}aley
  theory}, Math. Ann. \textbf{321} (2001), no.~4, 885--888. \MR{1872532}

\bibitem{Hunt}
Richard~A. Hunt, \emph{On the convergence of {F}ourier series}, Orthogonal
  {E}xpansions and their {C}ontinuous {A}nalogues ({P}roc. {C}onf.,
  {E}dwardsville, {I}ll., 1967), Southern Illinois Univ. Press, Carbondale,
  Ill., 1968, pp.~235--255. \MR{0238019}

\bibitem{Journe}
Jean-Lin Journ\'e, \emph{Calder\'on-{Z}ygmund operators on product spaces},
  Rev. Mat. Iberoamericana \textbf{1} (1985), no.~3, 55--91. \MR{836284}

\bibitem{Lacey_bilinearSqF}
Michael~T. Lacey, \emph{On bilinear {L}ittlewood-{P}aley square functions},
  Publ. Mat. \textbf{40} (1996), no.~2, 387--396. \MR{1425626}

\bibitem{Lacey_RdF}
\bysame, \emph{Issues related to {R}ubio de {F}rancia's {L}ittlewood-{P}aley
  inequality}, New York Journal of Mathematics. NYJM Monographs, vol.~2, State
  University of New York, University at Albany, Albany, NY, 2007. \MR{2293255}

\bibitem{MohantyShrivastava}
Parasar Mohanty and Saurabh Shrivastava, \emph{A note on the bilinear
  {L}ittlewood-{P}aley square function}, Proc. Amer. Math. Soc. \textbf{138}
  (2010), no.~6, 2095--2098. \MR{2596047}

\bibitem{MuscaluTaoThiele}
Camil Muscalu, Terence Tao, and Christoph Thiele, \emph{{$L^p$} estimates for
  the biest. {II}. {T}he {F}ourier case}, Math. Ann. \textbf{329} (2004),
  no.~3, 427--461. \MR{2127985}

\bibitem{OberlinSeegerTaoThieleWright}
Richard Oberlin, Andreas Seeger, Terence Tao, Christoph Thiele, and James
  Wright, \emph{A variation norm {C}arleson theorem}, J. Eur. Math. Soc. (JEMS)
  \textbf{14} (2012), no.~2, 421--464. \MR{2881301}

\bibitem{RubioDeFrancia}
Jos\'e~L. Rubio~de Francia, \emph{A {L}ittlewood-{P}aley inequality for
  arbitrary intervals}, Rev. Mat. Iberoamericana \textbf{1} (1985), no.~2,
  1--14. \MR{850681}

\bibitem{Silva}
Prabath Silva, \emph{Vector-valued inequalities for families of bilinear
  {H}ilbert transforms and applications to bi-parameter problems}, J. Lond.
  Math. Soc. (2) \textbf{90} (2014), no.~3, 695--724. \MR{3291796}

\bibitem{Sjoelin}
Per Sj\"{o}lin, \emph{A note on {L}ittlewood-{P}aley decompositions with
  arbitrary intervals}, J. Approx. Theory \textbf{48} (1986), no.~3, 328--334.
  \MR{864755}

\bibitem{Soria}
Fernando Soria, \emph{A note on a {L}ittlewood-{P}aley inequality for arbitrary
  intervals in {${\bf R}^2$}}, J. London Math. Soc. (2) \textbf{36} (1987),
  no.~1, 137--142. \MR{897682}

\end{thebibliography}
\bibliographystyle{amsplain}

\end{document}